\documentclass[reqno]{amsproc}
\usepackage{tikz}
\usepackage{amstext,amsmath,amssymb,amsfonts}
\usepackage[latin1]{inputenc}
\usepackage{epsfig}
\usepackage{hyperref}
\usepackage{color}

\usepackage{geometry}
\usepackage{amsthm}

\definecolor{myurlcolor}{rgb}{0,0,0.7}

\newtheorem{rmq}{Remark}[section]

\newtheorem{lem}{Lemma}[section]
\newtheorem{thm}{Theorem}[section]

\newtheorem{Ass}{Assumption}[section]
\newcommand{\bprof}{\begin{prof}}
\newcommand{\eprof}{\end{prof}}
\newenvironment{prof}[1][Proof]{\textbf{#1.} }{\ \rule{0.5em}{0.5em}}
\usepackage{cite}

\newcommand{\bea}{\begin{eqnarray}}
\newcommand{\eea}{\end{eqnarray}}
\newcommand{\beq}{\begin{equation}}
\newcommand{\eeq}{\end{equation}}
\newcommand{\enn}{\nonumber \end{equation}}
\newcommand{\beqs}{\begin{eqnarray*}}
\newcommand{\eeqs}{\end{eqnarray*}}

 \newcommand{\cE}{\mathcal{E}}
\newcommand{\cT}{\mathcal{T}}

\def\cN{{\mathcal N}}

\newcommand{\curl}{\mathop{\rm curl}\nolimits}
\newcommand{\dive}{\mathop{\rm div}\nolimits}
\def\bv{\textbf{v}}

\title[A new unified stabilized mixed finite element method]
{A new unified stabilized mixed finite element method of the Stokes-Darcy coupled problem: Isotropic discretization}

\author{ Hou\'edanou Koffi Wilfrid $^{(a)}$}
\email{a) khouedanou@yahoo.fr}
\address{D\'epartement de Math\'ematiques,
Universit\'e d'Abomey-Calavi (UAC), Rep. of Benin}

\begin{document}

\maketitle
\begin{Large}
\begin{abstract}\Large
In this paper we develop an a priori error analysis of a new unified mixed finite element method for the
coupling of fluid flow with porous media flow in $\mathbb{R}^N$, $N\in\{2,3\}$ on isotropic meshes.
Flows are governed by the Stokes and Darcy equations, respectively, and the corresponding transmission conditions are given by mass conservation, balance of normal forces, and the Beavers-Joseph-Saffman law.
The approach utilizes a modification of the Darcy problem which allows us to apply a variant nonconforming Crouzeix-Raviart finite element to the whole coupled Stokes-Darcy problem. 
The well-posedness of the finite element scheme and its convergence analysis are derived. Finally, the numerical experiments are presented, which confirm the excellent stability and accuracy of our method.
{\bf Mathematics Subject Classification [MSC]:} 74S05, 74S10, 74S15,
74S20, 74S25, 7430.
\end{abstract}

\tableofcontents
\section{Introduction}
There are many serious problems currently facing the world in which the coupling between groundwater and surface water is 
important. These include questions such as predicting how pollution discharges into streams, lakes, and rivers making its way into 
the water supply. This coupling is also important in technological applications involving filtration.
We  refer to the nice overview \cite{27} and the references therein for its physical background, modeling, and standard numerical 
methods. One important issue in the modeling of the coupled Darcy-Stokes flow is the treatement of the interface 
condition, where the  Stokes fluid meets the porous medium. In this paper, we only consider the so-called Beavers-Joseph-Saffman condition, which was experimentally derived by Beavers and Joseph in \cite{23}, modified by Saffman in \cite{44}, 
and later mathematically justified in \cite{32,17,48,37}.

It is well known that the discretization of the velocity and the pressure, for both Stokes and Darcy problems and the coupled of them, has to be made in a compatible way in order to avoid instabilities. Since, usually, stable elements for the free fluid flow cannot been successfully applied to the porous medium flow, most of the finite element formulations developed for the Stokes-Darcy coupled problem are based on appropriate combinations of stable elements for the Stokes equations with stable elements for the Darcy equations. In \cite{HA:2016,DI:09,AHN:15, 12, 18, 19, 7, 21, 45, 30, 32, 48, JM:2000,39,43,49,47,GL:2018,JAFH:2018,RJZY:2015, HJA:2017}, and in the references therein, we can find a large list of contributions 
devoted to numerically approximate the solution of this interaction problem, including conforming and nonconforming methods.

There are a lot of papers considering different finite element spaces
in each flow region (see, for example, \cite{29,46,47} and the references therein). In contrast to this, other articles use the same finite element spaces in both regions by, in general, introducing some penalizing terms (ref. for examples \cite{GL:2018, 7, AHN:15} and the references therein). 

In \cite{GL:2018}, a  conforming unified finite element has been proposed for the modified coupled Stokes-Darcy problem in a plane domain, which has simple and straightforward implementations. The authors apply the classical Mini-element to the whole coupled Stokes-Darcy coupled problem. An a priori error analysis is performed with some numerical tests confirming the convergence rates.

In this article, we
propose a modification of the Darcy problem which allows us to apply a variant nonconforming finite element to the whole coupled Stokes-Darcy problem. We use a variant nonconforming Crouzeix-Raviart finite element method that has so many advantages for the velocities and piecewise constant for the pressures in both the Stokes and Darcy regions, and apply a stabilization term penalizing the jumps over the element edges of the piecewise continuous velocities. We prove that the formulation satisfies the discrete inf-sup conditions, obtaining as a result optimal accuracy with respect to solution regularity. Numerical experimants are also presented, which confirm the excellent stability and optimal performance of our method. The difference between our paper and the reference \cite{GL:2018} is that our discretization is nonconforming in both the Stokes domain and Darcy domain (in $ \Omega\subset\mathbb{R}^N$, $N=2 \mbox{ or } 3$). As a result, additional terms are included in the priori error analysis that measure the non-conformity of the method. One essential difficulty in choosing the unified discretization is that, the Stokes side velocity is in $H^1$ while the Darcy side velocity is only in $H(\dive)$. Thus, we introduce a variant of the nonconforming Crouzeix-Raviart piecewise linear finite element space (larger than the space $\textbf{H}_h$ used in \cite{7}). 
The choice of $\textbf{H}_h$ [see (\ref{espacevitesse})] is more natural than the one introduced in \cite{7} since the space $\textbf{H}_h$
approximates only $H(\dive, \Omega_d)$ and not  $[H^1(\Omega_d)]^N$, while our a priori error analysis is only valid in this larger space.

The rest of the paper is organized as follows. In Section \ref{preliminaire} we present the modified coupled Stokes-Darcy problem in $\Omega\subset\mathbb{R}^N$, $N=2\mbox{ or } 3$, notations and the weak formulation. Section \ref{priori} is devoted to the finite element discretization and the error estimation. Finally, in Section \ref{test}, we present the results of numerical experiments to verify the predicted rates of convergence. 

\section{Preliminaries and notation}\label{preliminaire}
\subsection{Model problem}\label{soussection1}
We consider the model of a flow in a bounded domain $\Omega\subset \mathbb{R}^N$ $(N=2 \mbox{  or  } 3)$, consisting of a 
porous medium domain $\Omega_d$, where the flow is a Darcy flow, and an  open region $\Omega_s=\Omega\smallsetminus 
\overline{\Omega}_d,$ where the flow is governed by the Stokes equations. The two regions are separated by an interface 
$\Gamma_I=\partial \Omega_d\cap \partial \Omega_s.$ Let $\Gamma_l=\partial \Omega_l\smallsetminus \Gamma_I$, $l=s,d$.
Each interface and boundary is assumed to be polygonal $(N=2)$ or polyhedral 
$(N=3)$. We denote by $\textbf{n}_s$ (resp. $\textbf{n}_d$) the 
unit outward normal vector along $\partial \Omega_s$  (resp. $\partial \Omega_d$). Note that on the interface 
$\Gamma_I$, we have $\textbf{n}_s=-\textbf{n}_d$. 
The Figures \ref{Fc1} and \ref{Fc2}  give a schematic representation of the geometry.
\begin{figure}[http]\label{gdomaine1}
	\begin{minipage}[c]{.49\textwidth}
		\centering
		\begin{center}
			\tikzstyle{grisEncadre}=[thick, dashed, fill=gray!20]
			\begin{tikzpicture}[scale=0.85]
			color=gray!100;
			\draw [very thick](1,1)--(7,1);
			\draw [very thick](1,1)--(1,5.5);
			\draw [very thick](1,5.5)--(7,5.5);
			\draw [very thick](7,1)--(7,5.5);
			\draw [very thick](1,3)--(7,3);
			\draw [black,fill=gray!30] (1,1) -- (7,1) -- (7,3) --(1,3) -- cycle;
			\draw (3.7,1.6) node [above]{$\mbox{ \small\small $\Omega_d$:  \textbf{\small \small Porous Medium} }$};
			\draw (3.7,4) node [above]{$\mbox{  $\Omega_s$:  \textbf{\small \small Fluid Region} }$};
			\draw [>=stealth,->] [line width=1pt](2,3)--(2,3.5) node [right]{$\textbf{n}_d$};
			\draw [>=stealth,->] [line width=1pt](6,3)--(6,2.5) node [right]{$\textbf{n}_s$};
			\draw [>=stealth,->] [line width=1pt](4,2.8)--(4.7,2.8) node [below]{$\tau_j$};
			\draw  (4.5,3.8) node [below]{$\Gamma_I$};
			\draw[line width=0.5pt](1,1)--(1,3) node[midway,above,sloped]{$\Gamma_d$};
			\draw[line width=0.5pt](1,1)--(7,1) node[midway,below,sloped]{$\Gamma_d$};
			\draw[line width=0.5pt](7,1)--(7,3) node[midway,below,sloped]{$\Gamma_d$};
			\draw[line width=0.5pt](1,3)--(1,5.5) node[midway,above,sloped]{$\Gamma_s$};
			\draw[line width=0.5pt](1,5.5)--(7,5.5) node[midway,above,sloped]{$\Gamma_s$};
			\draw[line width=0.5pt](7,3)--(7,5.5) node[midway,below,sloped]{$\Gamma_s$};
			\end{tikzpicture}
		\end{center}
		\caption{\footnotesize{A sketch of the geometry of the problem (case: $\partial \Omega_d\neq \Gamma_I$)}}
		\label{Fc1}
	\end{minipage}
	\begin{minipage}[c]{.45\textwidth}
		\centering
		\begin{center}
			\tikzstyle{grisEncadre}=[thick, dashed, fill=gray!20]
			\begin{tikzpicture}[scale=0.85]
			color=gray!100;
			\draw [very thick](1,1)--(7,1);
			\draw [very thick](1,1)--(1,5.5)node[midway,above,sloped]{$\Gamma_s$};
			\draw [very thick](1,5.5)--(7,5.5);
			\draw [very thick](7,1)--(7,5.5);
			\draw (2,2)--(2,4)node[midway,above,sloped]{$\Gamma_d$};
			\draw (2,2)--(6,2);
			\draw (6,2)--(6,4);
			\draw (2,4)--(6,4);
			\draw (4,4.5) node [above]{$\mbox{  $\Omega_s$:  \textbf{\small \small Fluid Region} }$};
			\draw [black,fill=gray!30] (2,2) -- (2,4) -- (6,4) --(6,2) -- cycle;
			\draw (4.,3.2) node {\small \small $\Omega_d:\mbox{\textbf{\small\small \small Porous Medium}}$};
			\draw [>=stealth,->] [line width=1pt](4,2)--(4,2.5) node [right]{$\textbf{n}_s$};
			\draw [>=stealth,->] [line width=1pt](4,1.8)--(4.7,1.8) node [below]{$\tau_j$};
			\end{tikzpicture}
		\end{center}
		\caption{\footnotesize{ A sketch of the geometry of the problem (case $\partial\Omega_d=\Gamma_I.$)}}
		\label{Fc2}
	\end{minipage}
\end{figure}\mbox{  }

For any function $v$ defined in 
$\Omega$, since its restriction to  $\Omega_s$ or to $\Omega_d$ could play a different mathematical 
roles (for instance their traces on $\Gamma_I$), we will set
$v_s=v_{|\Omega_s}$ and $v_d=v_{|\Omega_d}$.

In  $\Omega$, we denote   by $\textbf{u}$ the fluid velocity and  by $p$ the pressure.
The motion of the  fluid  in $\Omega_s$ is described by the Stokes equations
\begin{eqnarray}\label{I1}
\left\{
\begin{array}{ccccccccc}\label{r}
-2\mu\dive \textbf{D}(\textbf{u})+ \nabla p &=&\textbf{f} &\mbox{ in }&  &\Omega_s,&\\
\dive \textbf{u}&=&g  &\mbox{  in   }&  &\Omega_s,&\\
\textbf{u}&=&\textbf{0}  &\mbox{  on  }&  &\Gamma_s,&
\end{array}
\right.
\end{eqnarray}
while in the porous medium $\Omega_d$, by Darcy's law
\begin{eqnarray}\label{I2}
\left\{
\begin{array}{cccccccccccc}\label{rDarcy}
\mu \textbf{K}^{-1}\mathbf{u}+\nabla p
&=&\textbf{f}  &\mbox{ in }&   &\Omega_d,&\\
\dive \textbf{u}&=& g  &\mbox{ in  }&  &\Omega_d,& \\
\textbf{u}\cdot\textbf{n} _d& =& 0  &\mbox{  on }&   &\Gamma_d.&
\end{array}
\right.
\end{eqnarray}
Here, $\mu> 0$ is the fluid viscosity, $\textbf{D}$ the deformation rate tensor   defined by 
\begin{eqnarray*}
	\textbf{D}(\psi)_{ij}:=\frac{1}{2} \left(\frac{\partial \psi_i}{\partial x_j}+
	\frac{\partial \psi_j}{\partial x_i}\right), \mbox{   } 1\leqslant i,j\leqslant N,
\end{eqnarray*} 
and $\textbf{K}$  a symmetric and uniformly positive definite tensor representing the rock permeability and satisfying, for 
some constants 
$0< K_*\leqslant K^*< +\infty$, 
$$K_*\xi^{T}\xi\leqslant \xi^{T}\textbf{K}(x)\xi\leqslant K^*\xi^{T}\xi, \mbox{     }
\forall x\in\Omega_d, \mbox{  } \xi\in \mathbb{R}^N.$$
$\textbf{f}\in [L^2(\Omega)]^N$ is a term related to body forces and  $g\in L^2(\Omega)$ a
source or sink term satisfying the compatibility condition 
\begin{eqnarray*}
	\int_{\Omega} g(x) dx =0.
\end{eqnarray*}
Finally we consider the following  interface conditions on $\Gamma_I:$
\begin{eqnarray} \label{cd1}
\textbf{u}_s\cdot \textbf{n}_s+\mathbf{u}_d\cdot \textbf{n}_d&=&0, \\
\label{cd2}
p_s-2\mu \textbf{n}_s\cdot\textbf{D}(\textbf{u}_s)\cdot \textbf{n}_s&=&p_d,\\
\label{cd3}
\frac{\sqrt{\kappa_j}}{\alpha_1} 2\textbf{n}_s\cdot\textbf{D}(\textbf{u}_s)\cdot\tau_j&=&-\textbf{u}_s\cdot\tau_j, \mbox{  } j=1,\ldots,N-1.
\end{eqnarray}
Here,  Eq. (\ref{cd1}) represents mass conservation, Eq. (\ref{cd2}) the balance of normal forces, and Eq. (\ref{cd3}) the
Beavers-Joseph-Saffman conditions. Moreover, $\{\tau_j\}_{j=1,\ldots,N-1}$ denotes an orthonormal system of tangent vectors
on $\Gamma_I$, $\kappa_j=\tau_j\cdot\textbf{K}\cdot\tau_j$, and $\alpha_1$ is a parameter determined by experimental evidence. 

Eqs. (\ref{I1}) to (\ref{cd3}) consist of the model of the coupled Stokes and Darcy flows problem that we will study below.
\subsection{New weak formulation}
We begin this subsection by introducing some useful notations. If $W$ is a bounded domain of $\mathbb{R}^N$ and $m$ is a non negative integer, the Sobolev space $H^m(W)=W^{m,2}(W)$ is 
defined in the usual way with the usual norm $\parallel\cdot\parallel_{m,W}$ and semi-norm $|\cdot|_{m,W}$. In particular, $H^0(W)=L^2(W)$ and we write $\parallel\cdot\parallel_W$ for $\parallel\cdot\parallel_{0,W}$. Similarly we   denote by $(\cdot,\cdot)_{W}$  the $L^2(W)$ $[L^2(W)]^N$ or $ [L^2(W)]^{N\times N}$ inner product. For shortness if $W$ is equal to $\Omega$, we will drop  the index $\Omega$, while  for any $m\geq 0$, 
$\parallel\cdot\parallel_{m,l}=\parallel\cdot\parallel_{m,\Omega_l}$, $|\cdot|_{m,l}=|\cdot|_{m,\Omega_l}$ 
and $(.,.)_l=(\cdot,\cdot)_{\Omega_l}$, for $l=s,d$.
The space  $H_0^m(\Omega)$ denotes the closure of $C_0^{\infty}(\Omega)$ in $H^{m}(\Omega)$. Let $[H^m(\Omega)]^N$ be the space of
vector valued functions $\textbf{v}=(v_1,\ldots,v_N)$ with components  $v_i$ in $H^m(\Omega)$. The 
norm and the seminorm on $[H^m(\Omega)]^N$ are given by 
\begin{eqnarray}
\parallel\textbf{v}\parallel_{m,\Omega}\hspace*{0.2cm}:=\hspace*{0.2cm}\left(\sum_{i=0}^N\parallel v_i\parallel_{m,\Omega }^2\right)^{1/2} 
\mbox{  and    } \hspace*{0.2cm}|\textbf{v}|_{m,\Omega}\hspace*{0.2cm}:=\hspace*{0.2cm}\left(\sum_{i=0}^N |v_i|_{m,\Omega}^2\right)^{1/2}.
\end{eqnarray}
For a connected open subset of the boundary $\Gamma\subset \partial \Omega_s\cup\partial \Omega_d$, we write 
$\langle.,.\rangle_{\Gamma}$ for the $L^2(\Gamma)$ inner product (or duality pairing), that is, for scalar 
valued functions $\lambda$, 
$\eta$ one defines:
\begin{eqnarray}
\langle\lambda,\eta\rangle_{\Gamma}&:=&\int_{\Gamma} \lambda(s)\cdot\eta(s) ds
\end{eqnarray}

We also define  the special vector-valued functions space 
\begin{eqnarray}
\textbf{H}(\dive,\Omega)&:=&\left\{\textbf{v}\in [L^2(\Omega)]^N: \dive \textbf{v} \in L^2(\Omega)\right\}
\end{eqnarray}

To give the variational formulation of our coupled problem we define the following two spaces for the velocity and the pressure:
\begin{eqnarray*}
	\textbf{H}&:=&\left\{\textbf{v}\in \textbf{H}(\dive,\Omega): \textbf{v}_s\in [H^1(\Omega_s)]^N,
	\textbf{  } \textbf{v}=\textbf{0} \mbox{  on  } \Gamma_s\mbox{  and   } \textbf{v}\cdot\textbf{n}_d=0 \mbox{  on  }\Gamma_d\right\}
\end{eqnarray*}
equipped with the norm 
\begin{eqnarray}
\parallel\textbf{v}\parallel_{\textbf{H}}&:=&\left(|\textbf{v}|_{1,s}^2+\parallel\textbf{v}\parallel_{d}^2+
\parallel\dive \textbf{v}\parallel_{d}^2\right)^{1/2},
\end{eqnarray}
and 
\begin{eqnarray}
Q=L_0^2(\Omega):=\left\{q\in L^2(\Omega): \int_{\Omega} q(x) dx=0\right\}.
\end{eqnarray}

Multiplying the first equation of (\ref{I1}) by a test fonction $\bf v\in \bf H$ and the second one by $q\in Q$, integrating by parts over $\Omega_s$ the terms involving $\dive \textbf{D}(\textbf{u})$ and $\nabla p$, yield the variational form of Stokes equations:
\begin{eqnarray}\nonumber\label{F9}
(\textbf{f}_s,\textbf{v}_s)_{\Omega_s}&=&2\mu\left(\textbf{D}(\textbf{u}_s),\textbf{D}(\textbf{v}_s)\right)_{\Omega_s}-
\left(p_s,\mbox{div}\mbox{ }\textbf{v}_s\right)_{\Omega_s}\\
&+&\left(\left\{p_s-2\mu\textbf{n}_s\cdot\textbf{D}(\textbf{u}_s)\cdot\textbf{n}_s\right\},
\textbf{v}_s\cdot\textbf{n}_s\right)_{\Gamma_I}\\\nonumber
&+&\displaystyle\sum_{j=1}^{N-1}\left(-2\mu\textbf{n}_s\cdot\textbf{D}(\textbf{u}_s)\cdot\tau_j,\textbf{v}_s\cdot\tau_j
\right)_{\Gamma_I}\\
-\left(q_s,\mbox{div}\mbox{ }\textbf{u}_s\right)_{\Omega_s}&=&-(g_s,q_s)_{\Omega_s}
\end{eqnarray}
Using interface conditions $(\ref{cd2})$ and $(\ref{cd3})$ in (\ref{F9}), we obtain:
\begin{eqnarray}\nonumber
(\textbf{f}_s,\textbf{v})_{\Omega_s}&=&2\mu\left(\textbf{D}(\textbf{u}_s),\textbf{D}(\textbf{v})\right)_{\Omega_s}-
\left(p_s,\mbox{div}\mbox{ }\textbf{v}\right)_{\Omega_s}\\
&+& ( p_d,\textbf{v}_s\cdot\textbf{n}_s)_{\Gamma_I}+
\displaystyle\sum_{j=1}^{N-1}\frac{\mu\alpha_1}{\sqrt{k_j}}(\textbf{u}_s\cdot\tau_j,\textbf{v}_s\cdot\tau_j
)_{\Gamma_I}\mbox{     }\forall \textbf{v}\in\textbf{H}\\
-\left(q,\mbox{div}\mbox{ }\textbf{u}_s\right)_{\Omega_s}&=&-(g_s,q)_{\Omega_s}\mbox{    } \forall q\in Q
\end{eqnarray}
We apply a similar treatment to the Darcy equations by testing the first equation of (\ref{I2}) with a smooth fonction $\textbf{v}\in \textbf{H}$ and the second on by $q\in Q$, integrating by parts over $\Omega_d$ the terms involving $\nabla p_d$, yield the variational form of Darcy equations:
\begin{eqnarray}\label{F12}
\left(\mu\textbf{K}^{-1}\textbf{u}_d,\textbf{v}\right)_{\Omega_d}&=& \left(p_d,\mbox{div}\mbox{ }\textbf{v}\right)_{\Omega_d}
+(\textbf{f}_d,\textbf{v})_{\Omega_d}-
( p_d,\textbf{v}_d\cdot\textbf{n}_d)_{\Gamma_I} \mbox{    }\forall \textbf{v}\in \textbf{H}\\
\left(\mbox{div}\mbox{ }\textbf{u}_d,q\right)_{\Omega_d}&=&\left(g_d,q\right)_{\Omega_d}\mbox{   } 
\forall q\in Q
\end{eqnarray}
Now, incorporating the first boundary interface condition (\ref{cd1}) and taking into account that the vector valued functions in $\textbf{H}$ have (weakly) continuous normal components on $\Gamma_I$ (see 
  \cite[Theorem 2.5]{girault:86}), the mixed variational formulation of the coupled problem (\ref{I1})-(\ref{cd3}) can be stated as follows \cite{AHN:15}: Find $(\textbf{u},p)\in\textbf{H}\times Q$ that satisfies
  \begin{equation}\label{la1}
  \left\{
  \begin{array}{ccc}
  \textbf{a}(\textbf{u},\textbf{v})+\textbf{b}(\textbf{v},p)&=& L(\bv),
  \hspace*{1cm}\forall \textbf{v}\in \textbf{H},\\
  \textbf{b}(\textbf{u},q)&=& G (q),\hspace*{1cm} \forall q\in Q.
  \end{array}
  \right.
  \end{equation}
  where the bilinear forms $\textbf{a}(\cdot,\cdot)$ and $\textbf{b}(\cdot,\cdot)$ are defined on $\textbf{H}\times \textbf{H}$ and $\textbf{H}\times Q$, respectively, as:
\begin{eqnarray*}
	\textbf{a}(\textbf{u},\textbf{v})&:=&2\mu(\textbf{D}(\textbf{u}),\textbf{D}(\textbf{v}))_s+
	\sum_{j=1}^{N-1}\frac{\mu\alpha_1}{\sqrt{\kappa_j}}\left\langle \textbf{u}_s\cdot\tau_j,
	\textbf{v}_s\cdot\tau_j\right\rangle_{\Gamma_I}+\mu\left(\textbf{K}^{-1}\textbf{u},\textbf{v}\right)_d
	\\
	\textbf{b}(\textbf{v},q)&:=&-\left(q,\dive \textbf{v}\right)_{\Omega_s}-
\left(q,\dive \textbf{v}\right)_{\Omega_d}
	\end{eqnarray*}
By last, the linear forms $L$ and $G$ are defined as:
\begin{eqnarray*}
	L(\textbf{v}):=(\textbf{f},\textbf{v})_{\Omega_s}+(\textbf{f},\textbf{v})_{\Omega_d}
	\hspace*{2cm } \mbox{  and   } \hspace*{2cm} G(q):=-(g,q)_{\Omega_s}-(g,q)_{\Omega_d}.
\end{eqnarray*}

It is easy to prove that $\textbf{a}$ et $\textbf{b}$ are continuous, $\textbf{b}$ satisfies the continuous inf-sup condtion and $\textbf{a}$ is coercive on the null space of $\textbf{b}$. It is also clear that $F$ and $G$ are continuous and bounded. Then, using the classical theory of mixed methods (see, e.g., \cite[Theorem and Corollary 4.1 in Chapter I]{girault:86})  it follows the well-posedness of the continuous formulation (\ref{la1}) and so the following theorem holds \cite{AHN:15}:
\begin{thm} If $\textbf{f}\in [L^2(\Omega)]^N$ and $g\in L_0^2(\Omega)$, 
	there exists a unique solution $(\textbf{u},p)\in \textbf{H}\times Q$ to the problem (\ref{la1}).\\
\end{thm}
\begin{rmq}
Note that if $g$ is of mean zero, (\ref{la1}) directly implies that (\ref{I1}), (\ref{I2}) and (\ref{cd1}) hold ( 
the differential equations being understood in the distributional sense), while the interface conditions 
(\ref{cd2}) and (\ref{cd3}) are imposed in a weak sense. Also, we observe that the mixed variational formulation of the coupled problem (\ref{I1})-(\ref{cd3}) is equivalent to weak formulation (2.4) (and also (2.5) of \cite{12}), with the particularity that, in our case, for any $\textbf{v}\in \textbf{H}$, we have that 
$\left< \textbf{v}_s-\textbf{v}_d, \textbf{n}_sp_s\right>_{\Gamma_I}=0$.
\end{rmq}

Now we introduce a modification to the Darcy equation, with the purpose in mind of the development of a unified discretization for the coupled problem, that is, the Stokes and Darcy parts be discretized using the same finite element spaces. The modification that we apply to the Darcy equation follows the idea (same argument) given in \cite{GL:2018}. Indeed, we observe that taking the second equation of Darcy' problem (\ref{I2}) we can write, for any $\textbf{v}\in \textbf{H},$
\begin{equation}
	\int_{\Omega_d}(\dive \textbf{u}_d-g_d)\dive\textbf{v}=0.
	\end{equation}
Then, by adding this equation to the first equation of the variational form in (\ref{F12}), we get:
\begin{eqnarray}\label{Fm}
\left(\mu\textbf{K}^{-1}\textbf{u}_d,\textbf{v}\right)_{\Omega_d}+\left(\dive\textbf{u}_d,\dive \textbf{v}\right)_{\Omega_d}-\left(p_d,\dive\textbf{v}\right)_{\Omega_d}\\\nonumber
+( p_d,\textbf{v}_d\cdot\textbf{n}_d)_{\Gamma_I}&=&
(\textbf{f}_d,\textbf{v})_{\Omega_d}+(\dive \textbf{v},g_d)_{\Omega_d}
 \mbox{    }\forall \textbf{v}\in \textbf{H}\\
\left(\mbox{div}\mbox{ }\textbf{u}_d,q\right)_{\Omega_d}&=&\left(g_d,q\right)_{\Omega_d}\mbox{   } 
\forall q\in Q
\end{eqnarray}
From now on, we work with this modified variational form of Darcy equations.

In the same way that before, incorporating the boundary conditions (\ref{cd1}) and remambering that, since $\textbf{v}\in\textbf{H}$, it  was (weakly) continuous normal components on $\Gamma_I$, the variational form of the modified Stokes-Darcy problem can be written as follows:
Find $(\textbf{u},p)\in\textbf{H}\times Q$ satisfying
\begin{equation}\label{modified}
\left\{
\begin{array}{ccc}
\tilde{\textbf{a}}(\textbf{u},\textbf{v})+\textbf{b}(\textbf{v},p)&=& \tilde{L}(\bv),
\hspace*{1cm}\forall \textbf{v}\in \textbf{H},\\
\textbf{b}(\textbf{u},q)&=& G (q),\hspace*{1cm} \forall q\in Q.
\end{array}
\right.
\end{equation}
where the bilinear forms $\tilde{\textbf{a}} (\cdot,\cdot)$ and $\textbf{b}(\cdot,\cdot)$ are defined on $\textbf{H}\times \textbf{H}$, $\textbf{H}\times Q$, respectively, as:
$$\tilde{\textbf{a}}(\textbf{u},\textbf{v})= 
2\mu(\textbf{D}(\textbf{u}),\textbf{D}(\textbf{v}))_s+
\sum_{j=1}^{N-1}\frac{\mu\alpha_1}{\sqrt{\kappa_j}}\left\langle \textbf{u}_s\cdot\tau_j,
\textbf{v}_s\cdot\tau_j\right\rangle_{\Gamma_I}+\mu\left(\textbf{K}^{-1}\textbf{u},\textbf{v}\right)_d+\left(\dive \textbf{u}_d,\dive \textbf{v}\right)_{\Omega_d}
$$
and 
$$\textbf{b}(\textbf{v},q):=-\left(q,\dive \textbf{v}\right)_{\Omega_s}-
\left(q,\dive \textbf{v}\right)_{\Omega_d}.$$
By last, the linear forms $\tilde{L}$ and $G$ are defined as:
\begin{eqnarray*}
\tilde{L}(\textbf{v}):=(\textbf{f},\textbf{v})_{\Omega_s}+(\textbf{f},\textbf{v})_{\Omega_d}+(\dive \textbf{u},\dive\textbf{v})_{\Omega_D}
	\hspace*{1cm } \mbox{  and   } \hspace*{1cm} G(q):=-(g,q)_{\Omega_s}-(g,q)_{\Omega_d}.
\end{eqnarray*}
Then, applying the classical theory of mixed methods it follows the well-posedness of the continuous formulation (\ref{modified}).
\begin{thm}
There exists a unique $(\textbf{u},p)\in \textbf{H}\times Q$ solution to modified formulation (\ref{modified}).
In addition, there exists a positive constant $\tilde{C}$, depending on the continuous inf-sup condition constant for $\textbf{b}$, the coercivity constant for $\tilde{\textbf{a}}$ and the boundedness constants for $\tilde{\textbf{a}}$  and $\textbf{b}$, such that:
\begin{equation}
	\parallel\textbf{u}\parallel_{\textbf{H}}+\parallel p\parallel_{Q}\leq \tilde{C}
	\left(\parallel \textbf{f}_s\parallel_{\Omega_s}+\parallel\textbf{f}_d\parallel_{\Omega_d}+\parallel g_d\parallel_{\Omega_d}+\parallel g_s\parallel_{\Omega_s}\right).
\end{equation}
\end{thm}
We end this section with some notation.
In $2D$, the $\curl$ of a scalar function $w$ is given as usual by 
$\curl w:=(\frac{\partial w}{\partial x_2},-\frac{\partial w}{\partial x_1})^{\top}$ while in $3D$, the $\curl $ of a vector function $\textbf{w}$
is given as usual by $\curl \textbf{w}:=\nabla \times \textbf{w}$. Finally, let $\mathbb{P}^k$ be the space of polynomials of 
total degree not larger than $k$. In order to avoid excessive use of constants, the abbreviations $x\lesssim y$ and $x\sim y$ 
stand for $x\leqslant cy$ and $c_1x\leqslant y \leqslant c_2x$, respectively, with positive constants independent of $x$, $y$ or $\cT_h$.

 
\section{A priori error analysis}\label{priori}
\subsection{Finite element discretization}
In this subsection, we will use a variant of  the nonconforming Crouzeix-Raviart piecewise linear finite element approximation for the
velocity and  piecewise constant approximation for the pressure. 


Let $\left\{\mathcal{T}_h\right\}_{h> 0}$ be a family of triangulations of $\Omega$ with 
nondegenerate elements (i.e. triangles for $N=2$ and 
tetrahedrons for $N=3$). For any $T\in\mathcal{T}_h$, we denote by $h_T$ the diameter of $T$ and $\rho_T$ the diameter of the 
largest ball inscribed into $T$ and set 
\begin{eqnarray}
h=\max_{T\in\mathcal{T}_h} h_T, 
\mbox{   and   } \sigma_h=\max_{T\in\mathcal{T}_h } \frac{h_T}{2r_T} 
\end{eqnarray}
We assume that the family of triangulations is regular, in the sense that there exists $\sigma_0> 0$ such that
$\sigma_h\leqslant \sigma_0$,  for all  $h> 0$. We also assume   that the triangulation is conform with respect to the partition of $\Omega$ into $\Omega_s$ and $\Omega_d$, namely
each $T\in \mathcal{T}_h$ is  either in 
$\Omega_s$ or in  $\Omega_d$ (see Fig. \ref{isotropic}, \ref{adm1}, \ref{adm2}):


\begin{figure}[http]
	\begin{minipage}[c]{.30\textwidth}
			\centering
		\begin{center}
		\begin{tikzpicture}[scale=0.5]
		\draw (0,1)--(7,1);
		\draw (0,1)--(2,4.5);
		\draw (7,1)--(2,4.5);
		\draw [line width=0.75pt] [>=latex,<->](0,0.75)--(7,0.75)node [midway,below,sloped] {$\mbox{diam} (T)=h_T$};
		\draw  (2.45,2.43) circle (1.4);
		\draw (2.45,2.) node [above]{$\bullet$};
		\draw (2.9,1.5) node [above]{$r_T$};
		\draw  [line width=0.75pt](2.45,2.43)--(2.5,1);
		\end{tikzpicture}
	\end{center}
\caption{\footnotesize{\small\small{Isotropic element $T$ in $2d$.}}}
\label{isotropic}
	\end{minipage}
	\begin{minipage}[c]{.30\textwidth}
		\centering
		\begin{center}
			\begin{tikzpicture}[scale=0.5]
			\draw (0,3)--(0,-3);
			\draw (0,3)--(2,0);
			\draw (0,3)--(-2,0);
			\draw (-2,0)--(2,0);
			\draw (2,0)--(0,-3);
			\draw (-2,0)--(0,-3);
			\end{tikzpicture}
		\end{center}
		\caption{\footnotesize{\small\small Example of conforming mesh in $2d$}}
		\label{adm1}
	\end{minipage}
	\begin{minipage}[c]{.30\textwidth}
		\centering
		\begin{center}
			\begin{tikzpicture}[scale=0.5]
			\draw (0,3)--(0,0);
			\draw (0,3)--(2,0);
			\draw (0,3)--(-2,0);
			\draw (-2,0)--(2,0);
			\draw (2,0)--(0,-3);
			\draw (-2,0)--(0,-3);
			\draw (0,0)node {$\bullet$};
			\end{tikzpicture}
		\end{center}
		\caption{\footnotesize{\small\small Example of nonconforming mesh in $2d$ }}
		\label{adm2}
	\end{minipage}
\end{figure}

Let $\mathcal{T}_h^s$ and $\mathcal{T}_h^d$ be the corresponding induced triangulations of 
$\Omega_s$ and $\Omega_d$.
For any $T\in \mathcal{T}_h$, we denote by $\cE (T)$ (resp. $\cN(T))$
the set of its edges $(N=2)$ or 
faces $(N=3)$ (resp. vertices)  and set 
$\cE_h=\displaystyle\bigcup_{T\in\mathcal{T}_h} \cE(T)$, $\cN_h=\displaystyle\bigcup_{T\in\mathcal{T}_h} \cN(T)$. 
For $\mathcal{A}\subset \overline{\Omega}$ we define 
$$
\cE_h(\mathcal{A})=\left\{ E\in\cE_h: E\subset \mathcal{A}\right\}.
$$

Notice that 
$\cE_h$ can be split up in  the form 
\begin{eqnarray}\label{var}
\cE_h=\cE_h(\Omega_s^+)\cup 
\cE_h(\Omega_d)\cup \cE_h(\partial \Omega_d)
\end{eqnarray}
where
$\Omega_s^+=\Omega_s\cup \Gamma_s.$
Note that $\cE_h(\Gamma_I)$ is included in 
$\cE_h(\partial \Omega_d)$. 

With every edges $E\in\cE_h$, we associate a unit vector $\textbf{n}_E$ 
such that $\textbf{n}_E$ is orthogonal
to $E$ and equals to the unit exterior normal vector to $\partial \Omega$ if $E\subset \partial \Omega$. 
For any $E\in\cE_h$
and any piecewise continuous function $\varphi$, 
we denote by $[\varphi]_E$ its jump   across $E$ in the direction of $\textbf{n}_E$:
\begin{eqnarray*}
	[\varphi]_E(x):=
	\left\{
	\begin{array}{cccccc}\label{r}
		&\displaystyle\lim_{t\rightarrow 0+} \varphi(x+t\textbf{n}_E)-\lim_{t\rightarrow 0+} \varphi (x-t\textbf{n}_E) &
		&\mbox{for an interior edge/face $E$,}&\\
		&- \displaystyle\lim_{t\rightarrow 0+}\varphi (x-t\textbf{n}_E)& &\mbox{for a boundary edge/face $E$}&
	\end{array}
	\right.
\end{eqnarray*}

\begin{figure}[http]
	\begin{center}
		\begin{tikzpicture}
		\draw (-4,0.5)--(-2,0.5) node [midway] {$\bullet$};;
		\draw (-4,0.5)--(-4,2.5) node [midway] {$\bullet$};
		\draw (-4,2.5)--(-2,0.5) node [midway] {$\bullet$};;
		\draw (-4,2.5) node [left]{$a_3$};
		\draw (-4,0.5) node [left]{$a_1$};
		\draw (-2,0.5) node [below]{$a_2$};
		\draw [>=stealth,->](-4,0.5)--(-4,3.5) node[left] {$y$};
		\draw [>=stealth,->](-4,0.5)--(-1,0.5) node[below] {$x$};
		\end{tikzpicture}
	\end{center}
	\caption{\footnotesize{$\mathbb{P}^1$-nonconforming finite element $T$ in $2d$.}
	}
	\label{P1nonconforme}
\end{figure}
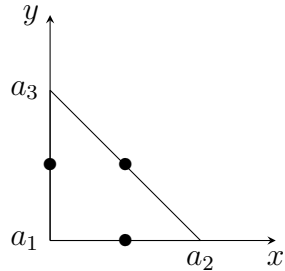\mbox{   }\\
 For $i\in\{0,\cdots,N\}$,  we set:
\begin{eqnarray}
\sigma_i(p)&:=& \frac{1}{|E_i|}\int_{E_i} p, \forall p\in \mathbb{P}^1(T), \mbox{ where } E_i\in\mathcal{E}(T)
\end{eqnarray} 
The triplet
$\{T,\mathbb{P}^1(T),\Sigma\}$ with  $\Sigma=\{\sigma_i\}_{0\hspace*{0.1cm}\leqslant i \hspace*{0.1cm}\leqslant N}$ is finite element \cite[Page 83]{AE05}. The local basis functions are defined by:
\begin{eqnarray}
\psi_i(T)&=& 1-N\lambda_i(T), \hspace*{0.3cm} i\in \{0,\ldots, N\},
\end{eqnarray}
where for each $i\in\{0,\cdots,N\}$, $\lambda_i(T)$ is barycentric coordonates of $T\in\cT_h$.

In classical reference element  $\overline{T}$, the basis fonctions are given by:
\begin{eqnarray}
\left\{
\begin{array}{cccccccccccc}
&\bar{\psi}_0(\bar{x},\bar{y})& &=& 1-2\bar{y},\\
&\bar{\psi}_1(\bar{x},\bar{y})& &=& -1+2\bar{x}+2\bar{y},\\
&\bar{\psi}_2(\bar{x},\bar{y})& &=& 1-2\bar{x}.
\end{array}
\right.
\end{eqnarray}

\begin{figure}[http]
	\begin{minipage}[c]{.20\textwidth}
		\centering
	\begin{center}
		\includegraphics[height= 4cm, width=4cm]{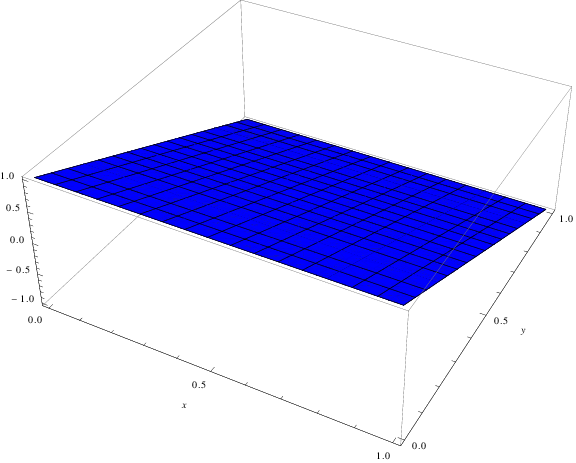}
	\end{center}
	\caption{\footnotesize{  $\bar{\psi}_0$.}}
\end{minipage}
\hspace*{2.5cm}
\begin{minipage}[c]{.20\textwidth}
	\centering
	\begin{center}
		\includegraphics[height= 4cm, width=4cm]{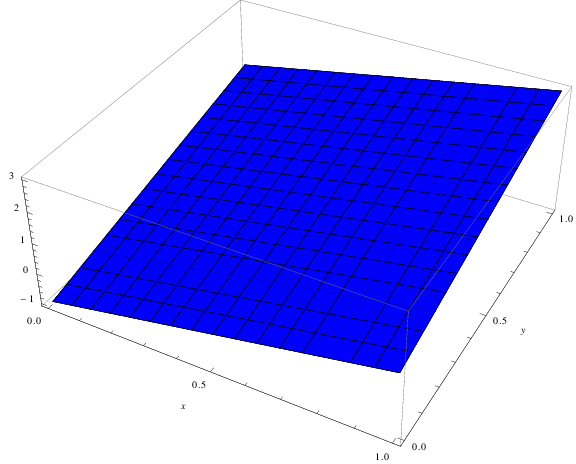}
	\end{center}
	\caption{\footnotesize{ $\bar{\psi}_1$.}}
\end{minipage}
\hspace*{2.5cm}
\begin{minipage}[c]{.20\textwidth}
	\centering
	\begin{center}
		\includegraphics[height= 4cm, width=4cm]{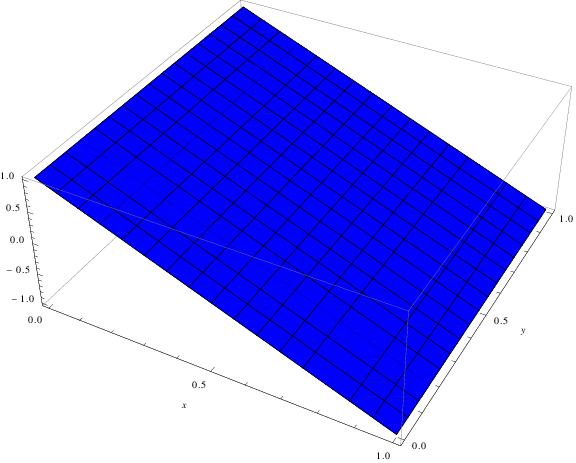}
	\end{center}
	\caption{\footnotesize{ $\bar{\psi}_2$.}}
	\end{minipage}
\end{figure}\mbox{  }\\

Based on the above notation, we introduce a variant of  the nonconforming Crouzeix-Raviart piecewise linear finite element space  (larger than  the space $\textbf{H}_h$ used in \cite{7})
\begin{eqnarray}\label{espacevitesse}
	\textbf{H}_h:=\left\{\textbf{v}_h: {\textbf{v}_h}_{|T}\in [\mathbb{P}^1(T)]^N \mbox{ } \forall T \in \mathcal{T}_h,
	([\textbf{v}_h]_E,\textbf{1})_E=0 \mbox{  } \forall E\in\cE_h(\Omega_s^+) \right.,\\
	\left. \left([\textbf{v}_h\cdot\textbf{n}_E]_E,1\right)_E=0 \hspace*{0.2cm}
	\forall E\in  
	\cE_h(\Omega_d)\cup\cE_h(\partial \Omega_d)\right\}
\end{eqnarray}
and piecewise constant function space 
\begin{eqnarray*}
	Q_h&:=&\left\{q_h\in L_0^2(\Omega): {q_h}_{|T}\in \mathbb{P}^0(T) \mbox{  } \forall T\in \mathcal{T}_h\right\},
\end{eqnarray*}
where $\mathbb{P}^m(T)$ is  the space of the restrictions to $T$ of all polynomials of degree less than or equal to $m$.
The space $Q_h$ is equipped with the norm $\parallel\cdot\parallel$ while the norm on $\textbf{H}_h$ will be specified later on.
The choice of $\textbf{H}_h$ is more natural than the one introduced in \cite{7} since the space $\textbf{H}_h$
approximates only $H(\dive, \Omega_d)$ and not  $[H^1(\Omega_d)]^N$, while our a priori error analysis is only valid in this larger space.

Let us introduce the discrete divergence operator 
$\dive_h\in \mathcal{L}(\textbf{H}_h;Q_h)\cap \mathcal{L}(\textbf{H};Q)$  by 
\begin{eqnarray}
(\dive_h \textbf{v}_h)_{|T}=\dive ({{\textbf{v}_h}_{|T}}), \forall T\in\mathcal{T}_h.
\end{eqnarray}
Then, we can introduce  two bilinear forms 
\begin{eqnarray*}
	\tilde{\textbf{a}}_h(\textbf{u},\textbf{v})&:=&2\mu\sum_{T\in\mathcal{T}_h^s}(\textbf{D}(\textbf{u}),\textbf{D}(\textbf{v}))_T+
	\sum_{j=1}^{N-1}\frac{\mu\alpha_1}{\sqrt{\kappa_j}}\langle \textbf{u}_s\cdot\tau_j,\textbf{v}_s\cdot\tau_j\rangle_{\Gamma_I}\\
	&+&\mu(\textbf{K}^{-1}\textbf{u},\textbf{v})_{\Omega_d}+\left(\dive_h\textbf{u},\dive_h\textbf{v}\right)_{\Omega_d}, \mbox{     }\forall \textbf{u},\textbf{v}\in \textbf{H}\cup\textbf{H}_h
\end{eqnarray*}
and 
\begin{eqnarray*}
	\textbf{b}_h(\textbf{v},q):=-(q,\dive_h \textbf{v})_{\Omega}, \mbox{    }\forall \textbf{v}\in \textbf{H}\cup
	\textbf{H}_h, \forall q\in Q_h.
\end{eqnarray*}
Then  the finite element discretization of (\ref{modified}) is to  find 
$(\textbf{u}_h,p_h)\in \textbf{H}_h\times Q_h$ such that 
\begin{eqnarray}\label{PD}
\left\{
\begin{array}{ccc}
\tilde{\textbf{a}}_h(\textbf{u}_h,\textbf{v}_h)+\textbf{b}_h(\textbf{v}_h,p_h)+
\textbf{J}(\textbf{u}_h,\textbf{v}_h)&=& \tilde{L}(\textbf{v}_h),
\forall \textbf{v}_h\in \textbf{H}_h,\\
\textbf{b}_h(\textbf{u}_h,q_h)&=&G(q_h), \forall q_h\in Q_h.
\end{array}
\right.
\end{eqnarray}
This is the natural discretization of the modified weak formulation (\ref{modified})  except that  the penalizing term 
$\textbf{J}(\textbf{u}_h,\textbf{v}_h)$ is added. This bilinear form $\textbf{J}(.,.)$ is defined by following the decomposition  
(\ref{dec}) of 
$\cE_h$: 
\begin{eqnarray}\label{dec}
\textbf{J}(\textbf{u},\textbf{v})=\textbf{J}_{\Omega_s^+}(\textbf{u},\textbf{v})+
\textbf{J}_{\Omega_d}(\textbf{u},\textbf{v})+\textbf{J}_{\partial \Omega_d}
(\textbf{u},\textbf{v})
\end{eqnarray}
where
\begin{eqnarray*}
	\textbf{J}_{\Omega_s^+}(\textbf{u},\textbf{v})&:=&(1+2\mu)
	\sum_{E\in \cE_h(\Omega_s^+)}h_E^{-1}\int_E[\textbf{u}]_E\cdot[\textbf{v}]_E ds,\\
	\textbf{J}_{\Omega_d}(\textbf{u},\textbf{v})&:=&\sum_{E\in \cE_h(\Omega_d)}h_E^{-1}
	\int_E[\textbf{u}]_E\cdot[\textbf{v}]_E ds, \hspace*{0.2cm}\mbox{    and    }\\
	\textbf{J}_{\partial \Omega_d}(\textbf{u},\textbf{v})&:=&\sum_{E\in\cE_h(\partial\Omega_d)}h_E^{-1}
	\int_E[\textbf{u}\cdot\textbf{n} _E]_E[\textbf{v}\cdot\textbf{n}_E]_E ds.
\end{eqnarray*}
Here, $h_E$ is the length ($N=2$) or diameter ($N=3$) of $E$.
Note that each element  of $\cE_h$ only contributes with one jump 
term in $\textbf{J}(\textbf{u},\textbf{v}).$
\begin{rmq}
	The Eq. (\ref{PD}) have the matrix representation
	\begin{eqnarray*}
	\textbf{M}_A\textbf{U}+\textbf{M}_B^T\textbf{P}+\textbf{M}_J\textbf{U}&=&\textbf{F}\\
	\textbf{M}_B\textbf{U}&=&\textbf{G}
	\end{eqnarray*}
where $\textbf{U}$ (resp. \textbf{P}) denote the  coefficients of $\textbf{u}_h$ (resp. $p_h$) expanded with respect to a basis for $\textbf{H}_h$ (rep. $Q_h$).
\end{rmq}

We are now able to define the norm on $\textbf{H}_h$ (see \cite{7}):
\begin{eqnarray*}
	\parallel\textbf{v}\parallel_h&:=&
	\left(\sum_{T\in\mathcal{T}_h^s}|\textbf{v}|_{1,T}^2+
	\sum_{j=1}^{N-1}\langle\textbf{v}_s\cdot\tau_j,\textbf{v}_s\cdot\tau_j\rangle_{\Gamma_I}+
	\parallel\textbf{v}\parallel_{\Omega_d}^2+\parallel\dive_h \textbf{v}\parallel_{\Omega_d}^2+
	\textbf{J}(\textbf{v},\textbf{v})\right)^{1/2}.
\end{eqnarray*}

In the sequel, we will denote by $\alpha$, $\beta$ and $C_i$ various constants independent of $h$. For the sake of convenience, we will define the bilinear form:
$$A_h(\textbf{u},\textbf{v})=\tilde{\textbf{a}}_h(\textbf{u},\textbf{v})+\textbf{J}(\textbf{u},\textbf{v}).$$
From H$\ddot{o}$lder's inequality, we derive the boundedness of $A_h(\cdot,\cdot)$ and $\textbf{b}_h(\cdot,\cdot)$:

\begin{lem} (Continuity of forms)
	There holds:
	\begin{eqnarray}
	|\tilde{L}(\textbf{v}_h)|&\leq & C_1\parallel \textbf{v}_h\parallel_{h}, \hspace*{0.4cm}\forall \textbf{v}_h\in \textbf{H}_h\cup
	\textbf{H},\\
	|G(q_h)|&\leq& C_2\parallel q_h\parallel, \hspace*{0.4cm}\forall q_h\in Q_h,\\
	|A_h(\textbf{u}_h,\textbf{v}_h)|&\leq& C_3\parallel\textbf{u}_h\parallel_h\times \parallel\textbf{v}_h\parallel_h,
	\hspace*{0.4cm} \forall \textbf{u}_h,\textbf{v}_h\in \textbf{H}_h\cup\textbf{H},\\
	|\textbf{b}_h(\textbf{v}_h,q_h)|&\leq& C_4\parallel\textbf{v}_h\parallel_h\times \parallel q_h\parallel, 
	\hspace*{0.4cm}\forall \textbf{v}_h\in \textbf{H}_h\cup\textbf{H}, \forall q_h\in Q_h.
	\end{eqnarray}
\end{lem}
\begin{thm}(Coercivity of $A_h$)\label{coercive}
There is an $\alpha> 0$ such that:
\begin{equation}\label{coercivity}
A_h(\textbf{v}_h,\textbf{v}_h)\geq \alpha \parallel\textbf{v}_h\parallel_h^2 \forall \textbf{v}_h\in \textbf{H}_h.
\end{equation}
\end{thm}
\begin{proof}
	Let $\textbf{v}_h\in\textbf{H}_h$. We have
	\begin{eqnarray*}
		A_h(\textbf{v}_h,\textbf{v}_h)&=&2\mu\displaystyle\sum_{T\in\mathcal{T}_h^s}\parallel\textbf{D}(\textbf{v}_h)\parallel_T^2+
		\mu(\textbf{K}^{-1}\textbf{v}_h,\textbf{v}_h)_{\Omega_d}+\displaystyle\sum_{j=1}^{N-1}\frac{\mu\alpha_1}{ \kappa_j} 
		\parallel \textbf{v}_h\cdot\tau_j\parallel_{\Gamma_I}^2\\
		&+& \parallel \dive_h \textbf{v}\parallel_{\Omega_d}^2+\textbf{J}_{\Omega_s^+}(\textbf{v}_h,\textbf{v}_h)+
		\textbf{J}_{\Omega_d}(\textbf{v}_h,\textbf{v}_h)+
		\textbf{J}_{\partial \Omega_d}(\textbf{v}_h,\textbf{v}_h)
	\end{eqnarray*}
We introduce the local space
\begin{eqnarray*}
H(\curl,T)&:=&
\left\{
\begin{array}{cccccccccccc}
	&\left\{\textbf{v}\in [L^2(T)]^2: \curl\textbf{v}\in L^2(T)\right\}&  &\mbox{  if  }& &N=2&,\\
	&\left\{\textbf{v}\in [L^2(T)]^3: \curl\textbf{v}\in [L^2(T)]^3\right\}&  &\mbox{  if  } & &N=3&.
\end{array}
\right.
\end{eqnarray*}
and for $\psi\in [H^1(T)]^N$, we define
\begin{eqnarray*}
	\gamma_{\tau}\psi:=
	\left\{
	\begin{array}{ccccccccc}
		&\psi\cdot\tau_{|\partial T}&   &\mbox{  if  } & & N=2, &\\
		&\psi\times \textbf{n}_{|\partial T}& &\mbox{  if }&  &N=3,& 
		(\tau\cdot\textbf{n} =  0 \mbox{ on } \partial T).
	\end{array}
	\right.
\end{eqnarray*}
with the semi-norm  
\begin{eqnarray}
\phi(\textbf{v}_h)&=&\left|\displaystyle\sum_{T\in\mathcal{T}_h^s}\int_T\curl \textbf{v}_h\right|_{\mathbb{R}^l}, 
\mbox{   } ( \mbox{ where } l=1 \mbox{ or  } l=3 ).
\end{eqnarray}
Using Young's inequality and Green formula, we have: 
\begin{eqnarray*}
	\phi(\textbf{v}_h)&=&\left|\int_{\Omega_s}\curl\textbf{v}_h\right|_{\mathbb{R}^l}\\
	&=&\left|\int_{\partial \Omega_s} \gamma_{\tau}(\textbf{v}_h)\right|_{\mathbb{R}^l} 
\end{eqnarray*}
\begin{eqnarray*}
	&=&\left|\int_{\Gamma_s} \gamma_{\tau}(\textbf{v}_h)\right|_{\mathbb{R}^l}+
	\left|\int_{\Gamma_I} \gamma_{\tau}(\textbf{v}_h)\right|_{\mathbb{R}^l}\\
	&\lesssim& \int_{\Gamma_s} \left|\gamma_{\tau}(\textbf{v}_h)\right|_{\mathbb{R}^l}+
	\int_{\Gamma_I} \left|\gamma_{\tau}(\textbf{v}_h)\right|_{\mathbb{R}^l}.
\end{eqnarray*}
$\bullet$ Estimate $\displaystyle\sum_{E\in \cE_h(\Gamma_s)}
\int_E\left| \gamma_{\tau}(\textbf{v}_h)\right|_{\mathbb{R}^l}$ ($l=1$ or $l=3$). We have by
 Cauchy-Schwarz inequality:
\begin{eqnarray*}
	\displaystyle\sum_{E\in \cE_h(\Gamma_s)}\int_E\left| \gamma_{\tau}(\textbf{v}_h)\right|_{\mathbb{R}^l}&\leqslant& 
	\displaystyle\sum_{E\in \cE_h(\Gamma_s)}\left\{\left( 
	\int_E\left| \gamma_{\tau}(\textbf{v}_h)\right|_{\mathbb{R}^l}^2\right)^{1/2}\times |h_E|^{1/2}\right\}\\
	&\leqslant& 
	\displaystyle\sum_{E\in \cE_h(\Gamma_s)} \left\{h_E^{-1/2}\times \left( 
	\int_E\left| \gamma_{\tau}(\textbf{v}_h)\right|_{\mathbb{R}^l}^2\right)^{1/2}\times h_E\right\}\\
	&\leqslant& 
	\left(\displaystyle\sum_{E\in \cE_h(\Gamma_s)} h_E^{-1}\int_E|[\textbf{v}_h]_E^2|_{\mathbb{R}^N}\right)^{1/2}\times 
	\left( \displaystyle\sum_{E\in \cE_h(\Gamma_s)} h_E^2\right)^{1/2}
\end{eqnarray*}
Also, we have:
\begin{eqnarray}
\left( \displaystyle\sum_{E\in \cE_h(\Gamma_s)} h_E^2\right)^{1/2} &\lesssim & 1,
\end{eqnarray}
Then,
\begin{eqnarray}
\displaystyle\sum_{E\in \cE_h(\Gamma_s)}\int_E\left| \gamma_{\tau}(\textbf{v}_h)\right|_{\mathbb{R}^l}&\lesssim& 
\left(\displaystyle\sum_{E\in \cE_h(\Gamma_s)} h_E^{-1}\int_E|[\textbf{v}_h]_E|_{\mathbb{R}^N}^2\right)^{1/2}.
\end{eqnarray}
Hence we deduce
\begin{eqnarray}
\displaystyle\sum_{E\in \cE_h(\Gamma_s)}\int_E\left| \gamma_{\tau}(\textbf{v}_h)\right|_{\mathbb{R}^l}
&\lesssim& \left(\textbf{J}_{\Omega_s^+} (\textbf{v}_h,\textbf{v}_h)\right)^{1/2}.
\end{eqnarray}
$\bullet$ Now we estime the term
$\displaystyle\sum_{E\in \cE_h(\Gamma_I)}\int_E\left| \gamma_{\tau}(\textbf{v}_h)\right|_{\mathbb{R}^l}$.
By Cauchy-Schwarz, we obtain:
\begin{eqnarray*}
	\displaystyle\sum_{E\in \cE_h(\Gamma_I)}\int_E\left| \gamma_{\tau}(\textbf{v}_h)\right|_{\mathbb{R}^l} &\leqslant&
	\left( \int_{\Gamma_I}|\gamma_{\tau}(\textbf{v}_h)|_{\mathbb{R}^l}^2\right)^{1/2}\times |\Gamma_I|^{1/2}
\end{eqnarray*}
\begin{eqnarray*}
	&\lesssim& \tilde{\textbf{a}}_h(\textbf{v}_h,\textbf{v}_h)^{1/2}.
\end{eqnarray*}
Thus we deduce the estimation:
\begin{eqnarray}
(\phi(\textbf{v}_h))^2 &\lesssim& \textbf{J}_{\Omega_s^+} (\textbf{v}_h,\textbf{v}_h)+
\tilde{\textbf{a}}_h(\textbf{v}_h,\textbf{v}_h).
\end{eqnarray}
Then,
\begin{eqnarray*}
	\textbf{J}_{\Omega_s^+} (\textbf{v}_h,\textbf{v}_h)+
	\textbf{a}_h(\textbf{v}_h,\textbf{v}_h)+\displaystyle\sum_{T\in\mathcal{T}_h^s}\parallel \textbf{D}(\textbf{v}_h)\parallel_T^2
	&\gtrsim& \displaystyle\sum_{T\in\mathcal{T}_h^s}\parallel \textbf{D}(\textbf{v}_h)\parallel_T^2+
	(\phi(\textbf{v}_h))^2+\\
	&+&
	\textbf{J}_{\Omega_s^+} (\textbf{v}_h,\textbf{v}_h).
\end{eqnarray*}
We apply Korn's discrete inequality \cite{B:03} and we get:
\begin{eqnarray}
\textbf{J}_{\Omega_s^+} (\textbf{v}_h,\textbf{v}_h)+
\textbf{a}_h(\textbf{v}_h,\textbf{v}_h)+\displaystyle\sum_{T\in\mathcal{T}_h^s}\parallel \textbf{D}(\textbf{v}_h)\parallel_T^2
&\gtrsim&
\displaystyle\sum_{T\in\mathcal{T}_h^s}\parallel \nabla(\textbf{v}_h)\parallel_{T}^2.
\end{eqnarray}
Thus
\begin{eqnarray*}
	\textbf{J}(\textbf{v}_h,\textbf{v}_h)+
	\tilde{\textbf{a}}_h(\textbf{v}_h,\textbf{v}_h)+\displaystyle\sum_{T\in\mathcal{T}_h^s}\parallel \textbf{D}(\textbf{v}_h)\parallel_T^2
	&\gtrsim&
	\displaystyle\sum_{T\in\mathcal{T}_h^s}\parallel \nabla(\textbf{v}_h)\parallel_{T}^2 +
	\textbf{J}(\textbf{v}_h,\textbf{v}_h),
\end{eqnarray*}
Hence,
\begin{eqnarray}\label{Coercivite1}
A_h(\textbf{v}_h,\textbf{v}_h)
&\gtrsim&
\displaystyle\sum_{T\in\mathcal{T}_h^s}\parallel \nabla(\textbf{v}_h)\parallel_{T}^2 +
\textbf{J}(\textbf{v}_h,\textbf{v}_h).
\end{eqnarray}
We have,
\begin{eqnarray}\label{coercivite2}
A_h(\textbf{v}_h,\textbf{v}_h)&\geqslant& \sum_{j=1}^{N-1}\parallel \textbf{v}_h\cdot\tau_j\parallel_{\Gamma_I}^2,\\
\label{coercivite3}
A_h(\textbf{v}_h,\textbf{v}_h)&\geqslant& \parallel \textbf{v}_h\parallel_{\Omega_d}^2\\
\label{Coercivite4}
A_h(\textbf{v}_h,\textbf{v}_h)&\geqslant& \parallel \dive_h\textbf{v}_h\parallel_{\Omega_d}^2
\end{eqnarray}
The estimates (\ref{Coercivite1}), (\ref{coercivite2}), (\ref{coercivite3}) and (\ref{Coercivite4}),
lead to (\ref{coercivity}). The proof is complete.
\end{proof}

In order to verify the discrete inf-sup condition, we define the space:
\begin{eqnarray}
\textbf{W}&:=&\left\{ \textbf{v}\in \textbf{H}: \mbox{  } \textbf{v}_{|\Omega_d}\in [H^1(\Omega_d)]^N\right\}.
\end{eqnarray}
We define also the Crouzeix-Raviart interpolation operator $\textbf{r}_h: \textbf{W}\rightarrow \textbf{H}_h$ by:

\begin{eqnarray}\label{Crouzei1}
\int_E(\textbf{r}_h\textbf{v})_s ds&=& \int_E\textbf{v}_s ds, \forall E\in \cE_h(\overline{\Omega}_s), \forall \textbf{v}\in \textbf{W},\\
\label{Crouzei2}
\int_E(\textbf{r}_h\textbf{v})_d ds&=& \int_E\textbf{v}_d ds, \forall E\in \cE_h(\overline{\Omega}_d), \forall \textbf{v}\in \textbf{W}.
\end{eqnarray}
\begin{lem}
The operator $\textbf{r}_h$ is bounded: there is a constant $C_5> 0$ depending on $\sigma$, $\mu$ and $N$ such that
\begin{eqnarray}
\parallel \textbf{r}_h\textbf{v}\parallel_h &\lesssim& \left(\parallel \textbf{v}\parallel_{1,s}^2+
\parallel \textbf{v}\parallel_{1,d}^2\right)^{1/2}, \forall \textbf{v}\in \textbf{W}.
\end{eqnarray}
\end{lem}
\begin{proof}
	The proof is similar to \cite{7}.
\end{proof}
Then, we have the following result
\begin{thm}(Discrete Inf-Sup condition)\label{Inf-Sup}
	 There exists a positive constant $\beta$ depending on $\sigma$, $\mu$ and $N$ such that
	\begin{eqnarray}\label{ISI}
	\displaystyle\inf_{q_h\in Q_h}\sup_{\textbf{v}_h\in \textbf{H}_h}
	\frac{\textbf{b}_h(\textbf{v}_h,q_h)}{\parallel \textbf{v}_h\parallel_h\parallel q_h\parallel} &\geq & \beta.
	\end{eqnarray}
	
\end{thm}
\begin{proof}
We use Fortin argument i.e. for each $q_h\in Q_h$, we find $\textbf{v}_h\in\textbf{H}_h$ such that:
\begin{eqnarray*}
	\textbf{b}_h(\textbf{v}_h,q_h)=\parallel q_h\parallel_{\Omega}^2 \mbox{  and   } \parallel\textbf{v}_h\parallel_h
	\hspace*{0.3cm}\lesssim\hspace*{0.3cm} 
	\parallel q_h\parallel_{\Omega}.
\end{eqnarray*}
Let $q_h\in Q_h\subset Q$. Then from 
\cite[Corollary 2.4, Page 24]{girault:86}, there exist vectoriel function
$\textbf{v}\in [H_0^1(\Omega)]^N$ satisfying
\begin{eqnarray}\label{IS1}
\left\{
\begin{array}{cccccccccccc}
\mbox{div}\mbox{ } \textbf{v}&=&-q_h, \mbox{  in } \Omega\\
\parallel \textbf{v}\parallel_{1,\Omega} &\lesssim& \parallel q_h\parallel_{\Omega}.
\end{array}
\right.
\end{eqnarray}
$[H_0^1(\Omega)]^N\subset W$, hence $\textbf{v}\in W$. We take $\textbf{v}_h=r_h\textbf{v}\in \textbf{H}_h$ and we have:
\begin{eqnarray*}
	\textbf{b}_h(\textbf{v}-r_h\textbf{v},q_h)&=&- \displaystyle\sum_{T\in\mathcal{T}_h}\int_{T} q_h\mbox{div}\mbox{ } (\textbf{v}-
	r_h\textbf{v}),\\
	&=&-\displaystyle\sum_{T\in\mathcal{T}_h}\int_{\partial T} q_h\textbf{n}_T\cdot (\textbf{v}-r_h\textbf{v})\\
	&=&-\displaystyle\sum_{E\in\cE_h(\Omega_s^+)}\int_E q_h\textbf{n}_E\cdot(\textbf{v}-r_h\textbf{v})\\
	&-&\displaystyle\sum_{E\in\cE_h(\overline{\Omega}_d)}\int_E q_h\textbf{n}_E\cdot(\textbf{v}-r_h\textbf{v})\\
	&=&0 \mbox{  }  (  \mbox{  from the identities  }  (\ref{Crouzei1}) \mbox{ and  } (\ref{Crouzei2})  ).
\end{eqnarray*}
Thus, we obtain
 \begin{eqnarray*}
	\textbf{b}_h(\textbf{v},q_h)&=&\textbf{b}_h(r_h\textbf{v},q_h).
\end{eqnarray*}
Using the system  $(\ref{IS1})$, we have:
\begin{eqnarray}\label{ISI2}
\textbf{b}_h(r_h\textbf{v},q_h)=-\int_{\Omega} q_h\mbox{div} (\textbf{v})=\parallel q_h\parallel^2 &\gtrsim&
\parallel \textbf{v}\parallel_{1,\Omega}\times \parallel q_h\parallel.
\end{eqnarray}
Also,
\begin{eqnarray}\label{ISI3}
\parallel \textbf{v}_h\parallel_h=\parallel r_h\textbf{v}\parallel_h\hspace*{0.2cm}\lesssim 
\hspace*{0.2cm}\parallel \textbf{v}\parallel_{1,\Omega}.
\end{eqnarray}
From (\ref{ISI2}) et (\ref{ISI3}), we deduce: 
\begin{eqnarray}
\textbf{b}_h(r_h\textbf{v},q_h)&\gtrsim& \parallel \textbf{v}_h\parallel_h\times \parallel q_h\parallel,
\mbox{  } \forall  q_h\in Q_h.
\end{eqnarray}
The Inf-Sup condition holds and the proof is complete.
\end{proof}

From Theorem \ref{coercive} and Theorem \ref{Inf-Sup} we have the following result:
\begin{thm}
There exists a unique solution $(\textbf{u}_h,p_h)\in \textbf{H}_h\times Q_h$ to the problem (\ref{modified}).
\end{thm}
\subsection{A convergence analysis}

We now present an a priori analysis of the approximation error: The use of nonconforming finite element leads to $\textbf{H}_h\nsubseteq \textbf{H}$, so the approximation error contains some extra consistency error terms. In fact, the abstract error estimates give the following result:
\begin{lem}\label{Strang}
Let $(\textbf{u},p)\in \textbf{H}\times Q$ be the solution of problem (\ref{modified}) and 
$(\textbf{u}_h,p_h)\in \textbf{H}_h\times Q_h$ be the solution of the discrete problem (\ref{PD}).
Then we have
\begin{eqnarray}
\parallel\textbf{u}-\textbf{u}_h\parallel_h+\parallel p-p_h\parallel\lesssim 
\displaystyle\inf_{\textbf{v}_h\in\textbf{H}_h}\parallel\textbf{u}-\textbf{v}_h\parallel_h+
\displaystyle\inf_{q_h\in Q_h}\parallel p-q_h\parallel+E_{1h}+E_{2h}.
\end{eqnarray}
where $E_{1h}$ and $E_{2h}$ are the consistency error terms define by:
\begin{eqnarray}
E_{1h}&=& \displaystyle\sup_{\textbf{v}_h\in\textbf{H}_h}
\frac{|A_h(\textbf{u},\textbf{v}_h)+\textbf{b}_h(\textbf{v}_h,p)-(\textbf{f},\textbf{v}_h)_{\Omega}-
	(\dive\textbf{u},\dive_h\textbf{v}_h)_{\Omega_d}|}
{\parallel\textbf{v}_h\parallel_h},\\
E_{2h}&=& \sup_{q_h\in Q_h}\frac{|\textbf{b}_h(\textbf{u},q_h)+(g,q_h)_{\Omega}|}{\parallel q_h\parallel}.
\end{eqnarray}
\end{lem}

Note that $\textbf{b}_h(\textbf{u},q_h)=\textbf{b}(\textbf{u},q_h)$, thus $E_{2h}=0$.

For estiming the approximation error, we assume that the solution $(\textbf{u},p)$ of problem (\ref{modified}) satisfies the smoothness assumptions:
\begin{Ass}\label{assumption}\mbox{   }
	\begin{enumerate}
	\item $\textbf{u}\in \textbf{H}$, $\textbf{u}_s\in [H^2(\Omega_s)]^N$, $\textbf{u}_d\in [H^2(\Omega_d)]^N$;
	\item $p\in Q$, $p_s\in H^1(\Omega_s)$, $p_d\in H^1(\Omega_d)$.
	\end{enumerate}
\end{Ass}
We begin with the estimates for the terms: 
$\displaystyle\inf_{\textbf{v}_h\in\textbf{H}_h}\parallel\textbf{u}-\textbf{v}_h\parallel_h \mbox{ and }
\displaystyle\inf_{q_h\in Q_h}\parallel p-q_h\parallel$
\begin{lem}(Ref. \cite{7}) \label{Leminterpolation} There hold:
	\begin{eqnarray}\label{Intpo1}
	\displaystyle\inf_{\textbf{v}_h\in\textbf{H}_h} \parallel 
	\textbf{u}-\textbf{v}_h\parallel_h&\lesssim& h\left( |\textbf{u}|_{2,s}+|\textbf{u}|_{2,d}\right),\\\label{Intpo2}
	\displaystyle\inf_{q_h\in Q_h}\parallel p-q_h\parallel &\lesssim& h\left( |p|_{1,s}+|p|_{1,d}\right).
	\end{eqnarray}
\end{lem}
Finally, let us consider the term $A_h(\textbf{u}_h,\textbf{v}_h)+\textbf{b}_h(\textbf{v}_h,p_h-\tilde{L}(\textbf{v}_h).$
The smoothness assumption of $\textbf{u}$ implies $\textbf{J}(\textbf{u},\textbf{v}_h)=0$, thus 
$A_h(\textbf{u},\textbf{v}_h)=\tilde{\textbf{a}}(\textbf{u},\textbf{v}_h), \forall \textbf{v}_h\in \textbf{H}_h$.
Clearly, 
\begin{eqnarray*}
	-\tilde{L}(\textbf{v}_h)&=&-(\textbf{f},\textbf{v}_h)_{\Omega}-(g,\dive_h\textbf{v}_h)_{\Omega_d}\\
	&=&-(\textbf{f},\textbf{v}_h)_{\Omega_s}-(\textbf{f},\textbf{v}_h)_{\Omega_d}
	-(g,\dive_h\textbf{v}_h)_{\Omega_d}\\
	&=&\left(2\mu\mbox{div}\mbox{ }\textbf{D}(\textbf{u})-\nabla p,\textbf{v}_h\right)_{\Omega_s}-\left(
	\mu\textbf{K}^{-1}\textbf{u}+\nabla p,\textbf{v}_h\right)_{\Omega_d}
-(\dive \textbf{u},\dive_h\textbf{v}_h)_{\Omega_d}	\\
	&=&2\mu\left(\mbox{div}\mbox{ } \textbf{D}(\textbf{u}),\textbf{v}_h\right)_{\Omega_s}
	-\left(\nabla p,\textbf{v}_h\right)_{\Omega_s}
	-\mu\left(\textbf{K}^{-1}\textbf{u},\textbf{v}_h\right)_{\Omega_d}-\left(\nabla p,\textbf{v}_h\right)_{\Omega_d}\\
	&-&(\dive \textbf{u},\dive_h\textbf{v}_h)_{\Omega_d}\\
	&=& \displaystyle\sum_{T\in\mathcal{T}_h^s}
	\Big\{2\mu\left(\mbox{div}\mbox{ } \textbf{D}(\textbf{u}),\textbf{v}_h\right)_{T}
	-\left(\nabla p,\textbf{v}_h\right)_{T}\Big\}+\\
	&+&\displaystyle\sum_{T\in\mathcal{T}_h^d}
	\Big\{\mu\left(\textbf{K}^{-1}\textbf{u},\textbf{v}_h\right)_{T}-
	\left(\nabla p,\textbf{v}_h\right)_{T} \Big\}-(\dive \textbf{u},\dive_h\textbf{v}_h)_{\Omega_d}
\end{eqnarray*}
\begin{eqnarray*}
	&=& 
	\displaystyle\sum_{T\in\mathcal{T}_h^s}
	\Big\{-2\mu \left(\textbf{D}(\textbf{u}\right),\textbf{D}(\textbf{v}_h))_T+2\mu
	\left(\textbf{n}_T\cdot\textbf{D}(\textbf{u}),\textbf{v}_h\right)_{\partial T}+
	\left(p,\mbox{div}\mbox{ } \textbf{v}_h\right)_{T}-\left(\textbf{v}_h\cdot\textbf{n}_T,p\right)_{\partial T}\Big\}
	\\
	&+&\displaystyle\sum_{T\in\mathcal{T}_h^d}
	\Big\{-\mu \left(\textbf{K}^{-1}\textbf{u},\textbf{v}_h\right)_T+\left(p,\mbox{div}\mbox{ }\textbf{v}_h\right)_T-
	\left(\textbf{v}_h\cdot\textbf{n}_T,p\right)_{\partial T}\Big\}
	-(\dive \textbf{u},\dive_h\textbf{v}_h)_{\Omega_d}\\
	&=& -\left\{\displaystyle \sum_{T\in \mathcal{T}_h^s} 2\mu\left(\textbf{D}(\textbf{u}),\textbf{D}(\textbf{v}_h)\right)_T\right\}
	+\mu\left(\textbf{K}^{-1}\textbf{u},\textbf{v}_h\right)_{\Omega_d}-\left(p,\dive_h\textbf{v}_h\right)_{\Omega} -(\dive \textbf{u},\dive_h\textbf{v}_h)_{\Omega_d}\\
	&+&\displaystyle\sum_{T\in\mathcal{T}_h^s}\Big\{2\mu(\textbf{n}_T\cdot\textbf{D}(\textbf{u}),\textbf{v}_h)_{\partial T}
	-\left( \textbf{v}_h\cdot\textbf{n}_T,p\right)_{\partial T}\Big\}-
	\displaystyle\sum_{T\in\mathcal{T}_h^d}(\textbf{v}_h\cdot\textbf{n}_T,p)_{\partial T}\\
	&=&-\tilde{\textbf{a}}_h(\textbf{u},\textbf{v}_h)+\sum_{j=1}^{N-1}\frac{\mu\alpha_1}{\sqrt{\kappa_j}}
	(\textbf{u}_s\cdot\tau_j,\textbf{v}_{h,s}\cdot\tau_j)_{\Gamma_I}-\textbf{b}_h(\textbf{u},\textbf{v}_h)\\ 
	&+&2\mu\displaystyle\sum_{E\in\cE_h(\Omega_s^+)}\left(\textbf{n}_E\cdot\textbf{D}(\textbf{u}),[\textbf{v}_h]_E\right)_E-
	\displaystyle\sum_{E\in\cE_h(\Omega_d)\cup\cE_h(\partial \Omega_d)}\left([\textbf{v}_h\cdot\textbf{n}_E]_E,p_d\right)_E\\
	&-&\displaystyle\sum_{E\in\cE_h(\Omega_s^+)}\left([\textbf{v}_h\cdot\textbf{n}_E,p_s]_E\right)_E.
\end{eqnarray*}
Thus, we have
\begin{eqnarray}
\tilde{\textbf{a}}(\textbf{u},\textbf{v}_h)+\textbf{J}(\textbf{u},\textbf{v}_h)+\textbf{b}_h(\textbf{v}_h,p)-\tilde{L}(\textbf{v}_h)=
R_1(\textbf{v}_h)+R_2(\textbf{v}_h)+R_3(\textbf{v}_h)+R_4(\textbf{v}_h),
\end{eqnarray}
where
\begin{eqnarray*}
	R_1 (\textbf{v}_h)&=&\sum_{j=1}^{N-1}\frac{\mu\alpha_1}{\sqrt{\kappa_j}}
	(\textbf{u}_s\cdot\tau_j,\textbf{v}_{h,s}\cdot\tau_j)_{\Gamma_I},\\
	R_2(\textbf{v}_h)&=&2\mu\displaystyle\sum_{E\in\cE_h(\Omega_s^+)}\left(\textbf{n}_E\cdot\textbf{D}(\textbf{u}),[\textbf{v}_h]_E\right)_E,\\
	R_3 (\textbf{v}_h)&=&\displaystyle\sum_{E\in\cE_h(\Omega_d)\cup\cE_h
		(\partial \Omega_d)}\left([\textbf{v}_h\cdot\textbf{n}_E]_E,p_d\right)_E,\\
	R_4 (\textbf{v}_h)&=& \displaystyle\sum_{E\in\cE_h(\Omega_s^+)}\left([\textbf{v}_h\cdot\textbf{n}_E,p_s]_E\right)_E.
\end{eqnarray*}
In order to evaluate the four face integrals, let us introduce two projections operators in the following. 

For any $T\in\cT_h$ and $E\in\cE(T)$, denote by $P_0(E)$ the constant space of the restrictions to $E$ and 
$\pi_E$ the projection operator from $L^(E)$ on to $P_0(E)$ such that
\begin{eqnarray}\label{interpolation1}
	\int_E \pi_E v=\int_E vds.
\end{eqnarray}
The operator $\pi_E$ has the property \cite{CR:73}:
\begin{eqnarray}\label{interpolation2}
\parallel v-\pi_E\parallel_{0,E}\lesssim h_E^{1/2}|v|_{1,T} \forall v\in H^1(T).
\end{eqnarray}

For any $\textbf{v}\in [L^2(E)]^N$, we let $\Pi_E\textbf{v}$ be the function in $[P_0(E)]^N$ such that 
$$(\Pi_E\textbf{v})_i=\pi_E v_i, 1\leq i\leq N.$$
Using inequality (\ref{interpolation2}), we obtain
\begin{eqnarray}\label{interpolation3}
\parallel \textbf{v}-\Pi_E\textbf{v}\parallel_{0,E}\lesssim h_E^{1/2}|\textbf{v}|_{1,T} \mbox{  } 
\forall \textbf{v}\in [H^1(T)]^N.
\end{eqnarray}
Then we have the following lemma:
\begin{lem}(Estimation the four face integrals)\label{Eface} There holds:
	\begin{eqnarray}\label{e1}
	|R_1(\textbf{v}_h)|&\leq& \max_{1\leqslant j\leqslant N-1}\left(\frac{\mu\alpha_1}{\sqrt{\kappa_j}}\right)h
	\parallel\textbf{u}_s\parallel_{1,\Omega_s}\parallel\textbf{v}_h\parallel_h\\\label{e2}
	|R_2(\textbf{v}_h)| &\lesssim& |\textbf{u}|_{2,s}\parallel \textbf{v}_h\parallel_h\\\label{e3}
	|R_3(\textbf{v}_h)|&\lesssim& h\left(|p|_{1,s}+|p|_{1,d}\right)\parallel\textbf{v}_h\parallel_h\\\label{e4}
	|R_4(\textbf{v}_h)|&\lesssim& h(|p|_{1,d})\parallel\textbf{v}_h\parallel_h.
	\end{eqnarray}
\end{lem}

\begin{proof}\mbox{  }
\begin{enumerate}
	\item Estimate (\ref{e1}): We begin with an estimate for the first term $R_1(\textbf{v}_h)$. For any 
	face $E\in\cE_h(\Omega_s^{+})$, there exists at least one element $T\in\cT_h^s$
	 such that $E\in\cE(T)$. Then, from condition (\ref{interpolation1}), H$\ddot{o}$der's inequality and inequality  (\ref{interpolation3}), it follows that
	\begin{eqnarray*}
		|R_1(\textbf{v}_h)|&\leqslant& \sum_{j=1}^{N-1}\frac{\mu\alpha_1}{\sqrt{\kappa_j}}\left(\int_{\Gamma_I} |\textbf{u}_s\cdot\tau_j|^2\right)^{1/2} \left(\int_{\Gamma_I}
		|\textbf{v}_{h,s}\cdot\tau_j|^2\right)^{1/2}\\
		&\leqslant& \sum_{j=1}^{N-1}\frac{\mu\alpha_1}{\sqrt{\kappa_j}}\parallel\textbf{u}_s\parallel_{1,\Omega_s}
		\left(\int_{\Gamma_I}|\textbf{v}_{h,s}\cdot\tau_j|^2\right)^{1/2}\\
		&\leqslant& \max_{1\leqslant j\leqslant N-1}\left(\frac{\mu\alpha_1}{\sqrt{\kappa_j}}\right)h
		\parallel\textbf{u}_s\parallel_{1,\Omega_s}\parallel\textbf{v}_h\parallel_h.
	\end{eqnarray*}
\item Estimate (\ref{e2}): \\ We have
$\textbf{n}_E\cdot\textbf{D}(\textbf{u})_{|E}\in [L^2(E)]^N$, hence
$\Pi_E(\textbf{n}_E\cdot\textbf{D}(\textbf{u}))\in P_0(E)^N$.\\
\begin{eqnarray}
\int_E\Pi_E(\textbf{n}_E\cdot  \textbf{D}(\textbf{u}))\cdot[\textbf{v}_h]_E&=& 
\Pi_E(\textbf{n}_E\cdot \textbf{D}(\textbf{u}))\int_E[\textbf{v}_h]_E=0.
\end{eqnarray}
Thus,
\begin{eqnarray*}
	\int_E\textbf{n}_E\cdot \textbf{D}(\textbf{u})\cdot[\textbf{v}_h]_E&=& 
	\int_E\left( \textbf{n}_E\cdot \textbf{D}(\textbf{u})-
	\Pi_E(\textbf{n}_E\cdot \textbf{D}(\textbf{u})\right)\cdot[\textbf{v}_h]_E\\
	&=& \int_E(I-\Pi_E)(\textbf{n}_E\cdot \textbf{D}(\textbf{u}))\cdot [\textbf{v}_h]_E\\
	&\lesssim&\parallel h_E^{1/2}(I-\Pi_E)(\textbf{n}_E)\cdot \textbf{D}(\textbf{u}) \parallel_E
	\parallel h_E^{-1/2}[\textbf{v}_h]_E\parallel_E\\
	&\lesssim& h_E |\textbf{D}(\textbf{u})|_{1,T}h_E^{-1/2} \parallel[\textbf{v}_h]_E\parallel_E.
\end{eqnarray*}
Furthermore, summing on $E\in \cE_h(\Omega_s^+)$ faces, we obtain the estimate:
\begin{eqnarray}
|R_2(\textbf{v}_h)|&\lesssim& h|\textbf{u}|_{2,s} \parallel \textbf{v}_h\parallel_h.
\end{eqnarray}
\item For the terms $R_3(\textbf{v}_h)$ and $R_4(\textbf{v}_h)$, we use the same techniques as in the proof of the bounds for $R_i(\textbf{v}_h)$, $i\in\{1,2\}$, and we obtain:
\begin{eqnarray*}
	|R_3(\textbf{v}_h)|&\lesssim& h\left( |p|_{1,s}+|p|_{1,d}\right)\parallel \textbf{v}_h\parallel_h,\\
	|R_4(\textbf{v}_h)|&\lesssim& h\left(|p|_{1,d}\right) \parallel \textbf{v}_h\parallel_h.
\end{eqnarray*}
\end{enumerate}
The proof is complete.
\end{proof}

From Lemma \ref{Strang}, Lemma \ref{Leminterpolation} and Lemma \ref{Eface}, now we derive the following convergence theorem:
\begin{thm}
Let the solution $(\textbf{u},p)$ of problem (\ref{modified}) satifies the smoothness assumption (Assumption \ref{assumption}). Let $(\textbf{u}_h,p_h)$ be the solution of the discrete problem (\ref{PD}). Then there exists a positive constant $C$ depending on $N$, $\mu$, $K_{*}$, $K^{*}$, $\alpha_1$ and $\sigma$ such that:
\begin{eqnarray}
\parallel \textbf{u}-\textbf{u}_h\parallel_h+\parallel p-p_h\parallel\leq
Ch\left(|\textbf{u}|_{2,s}+|\textbf{u}_{2,d}+|p|_{1,s}+|p|_{1,d}\right).
\end{eqnarray}
\end{thm}
\section{Numerical experiments}\label{test}
In this section we present one test case to verify the predicted rates of convergence. 
The numerical simulations have been performed on the finite element code FreeFem++ \cite{HP:FreeFem, FH:98} in isotropic coupled mesh of Fig. \ref{TriC}. The solutions have been represented by Mathematica software.
For simplicity we choose each domain $\Omega_l$, $l\in\{ s,d\}$ as the unit square, $\alpha_1=\mu=1$, and the permeability tensor $\textbf{K}$ is taken to be the identity. The interface $\Gamma_I$, is the line $x=1$, i.e. 
$\overline{\Omega}=[0,1[\cup\{1\}\cup ]1,2]$ like the show the Figure \ref{Domainenumerique}.
\begin{figure}[http]
	\begin{center}
		\begin{tikzpicture}[scale=1.5]
		\draw (0,0)--(2,0); 
		\draw (0,0) node [ left  ] {$(0,0)$};
		\draw (1.7,-0.4) node [ above right] {$(1,0)$};
		\draw (4,-0.1) node [ above right ] {$(2,0)$};
		\draw (4,2.) node [ above right] {$(2,1)$};
		\draw (1.8,2.) node [ above right ] {$(1,1)$};
		\draw (0,2.) node [ above left ] {$(0,1)$};
		\draw (2,0)--(4,0);
		\draw (4,0)--(4,2);
		\draw (4,2)--(2,2);
		\draw (2,2)--(2,0);
		\draw (2,2)--(0,2);
		\draw (0,0)--(0,2);
		\draw (2,2)--(2,0);
		\draw (1,1) node {$\Omega_s$};
		\draw (3,1) node {$\Omega_d$};
		\end{tikzpicture}
	\end{center}
	\caption{\footnotesize{The domain $\Omega$ in $2d$.}}
	\label{Domainenumerique}
\end{figure}
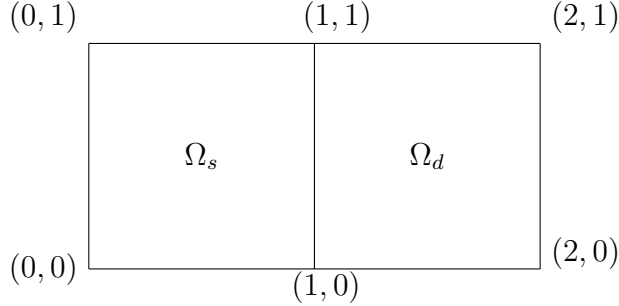\newline
We consider the application $\phi: (x,y)\in\mathbb{R}^2\longmapsto\phi(x,y)= x^2(x-1)^3y^2(y-1)^2\in\mathbb{R}$  on
the square  $\Omega=]0,1[^2\cup ]1,2[^2$. In $\Omega$, we define $\textbf{u}=(u_1,u_2)=\curl \phi 
\left(-\frac{\partial \phi}{\partial y}, \frac{\partial \phi}{\partial x}\right)$ and we obtain:
\begin{eqnarray}
u_1(x,y)&:=& -2(-1+x)^3x^2(-1+y)y(-1+2y)\\
u_2(x,y)&:=&(-1+x)^2x(-2+5x)(-1+y)^2y^2
\end{eqnarray}
We choose quadratic pressure $p\in L^2(\Omega)$ by
\begin{eqnarray}
p(x,y)&=& x^2-2xy+\frac{y^2}{2}-1.
\end{eqnarray}
Thus,
\begin{eqnarray}
\int_{\Omega} p(x,y) dxdy = 0 \hspace*{0.2cm} \mbox{  and   }\nabla p=\left(2x-2y,-2x+y\right).
\end{eqnarray}
The exact solution $(\textbf{u},p)$  satifies the following condition:
\begin{eqnarray}
\mbox{ div } \textbf{u}&=& 0 =g\mbox{  in  } \Omega,\\
\textbf{u}& = &\textbf{0} \mbox{  on } \partial \Omega,
\end{eqnarray}
and the Beavers-Joseph-Saffman interface conditions on $\Gamma_I$ [$\Gamma_I: x=1$]:
\begin{eqnarray}
\textbf{u}_s\cdot\textbf{n}_s+\textbf{u}_d\cdot\textbf{n}_d&=& 0\mbox{  on  } \Gamma_I,\\
p_s-2\mu\textbf{n}_s\cdot\textbf{D}(\textbf{u}_s)\cdot\textbf{n}_s&=&p_d\mbox{   on  }\Gamma_I,\\
\frac{\sqrt{k_j}}
{\alpha_1}2\textbf{n}_s\cdot\textbf{D}(\textbf{u}_s)\cdot\tau_j &=&-\textbf{u}_s\cdot\tau_j
\mbox{  on } \Gamma_I,
\mbox{    } j=1,\ldots,N-1.                                                                                                             
\end{eqnarray}
Furthermore, we obtain the right-hand terms $\textbf{f}$ define by

\begin{eqnarray}
\left\{
\begin{array}{ccccccccc}\label{r}
\textbf{f}_s &=& &-2\mu\dive \textbf{D}(\textbf{u})+ \nabla p& &\mbox{ in }&  &\Omega_s,&\\
\textbf{f}_d&=& &\mu \textbf{K}^{-1}\mathbf{u}+\nabla p& &\mbox{ in }&   &\Omega_d&.
\end{array}
\right.
\end{eqnarray}
Thus,  $\textbf{f}_s(x,y)=(f_1(x,y),f_2(x,y))$  in $\Omega_s$ leads to
\begin{eqnarray*}
	f_1(x,y)&=&4(-1+x)(-1+2y)(-6x^3+3x^4
	+(-1+y)y-8x(-1+y)y,\\
	&+&x^2(3+10(-1+y)y))+2x-2y,\\
	f_2(x,y)&=&-2\left(9(-1+y)^2y^2-12x^3(1+6(-1+y)y)+5x^4(1+6(-1+y)y)\right.\\
	&-&\left.2x(1+6(-1+y)y(1+3(-1+y)y))+x^2(9+6(-1+y)y(9+5(-1+y)y))\right)\\
	&-&2x+y,
\end{eqnarray*}
and in $\Omega_d$,  $\textbf{f}_d(x,y)=(k_1(x,y),k_2(x,y))$ is given by:
\begin{eqnarray*}
	k_1(x,y)&=& (-1+x)^2x(-2+5x)(-1+y)^2y^2+2x-2y,\\
	k_2(x,y)&=& (-1+x)^2x(-2+5x)(-1+y)^2y^2-2x+y.
\end{eqnarray*}
\mbox{}\\
\begin{figure}[http]
	\begin{minipage}[c]{.49\textwidth}
		\centering
		\begin{center}
			\includegraphics[width=1\textwidth,angle=0]{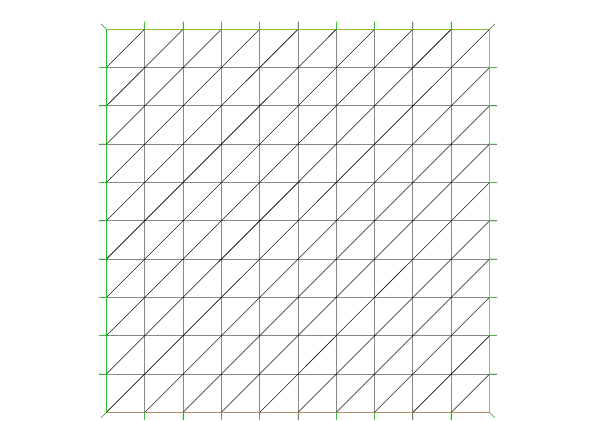}
		\end{center}
		\caption{\footnotesize{Example of isotropic mesh in $2d$}}
	\end{minipage}
	\begin{minipage}[c]{.49\textwidth}
		\centering
		\begin{center}
			\includegraphics[width=1\textwidth]{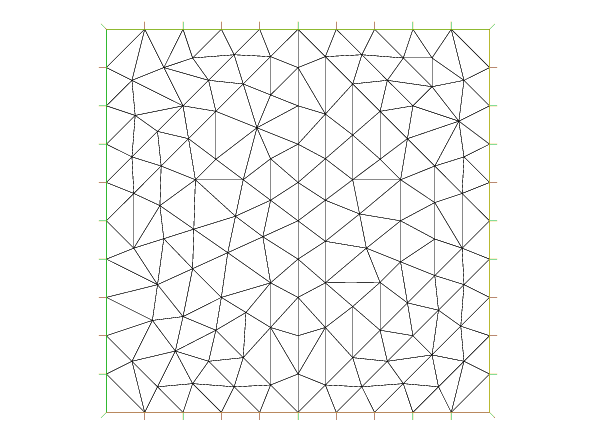}
		\end{center}
		\caption{\footnotesize{Example of anisotropic mesh in $2d$}}
	\end{minipage}
\end{figure}
\begin{figure}[http]
	\begin{center}
		\includegraphics[width=1.\textwidth]{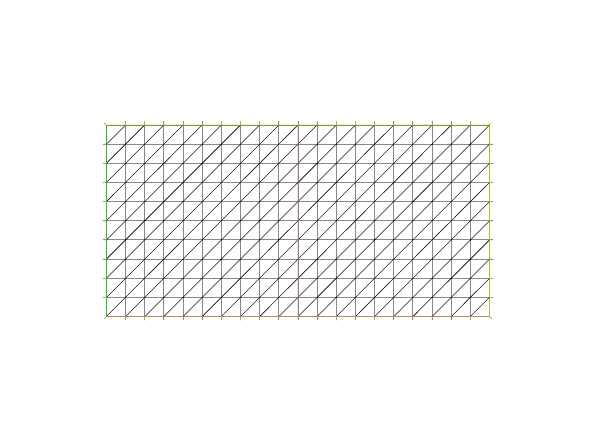}
		\caption{\footnotesize{Isotropic mesh on coupled domain $\Omega\subset\mathbb{R}^2$}}
		\label{TriC}
	\end{center}
\end{figure}

\begin{figure}[http]
	\begin{minipage}[c]{.49\textwidth}
		\centering
		\begin{center}
			\includegraphics[width=.6\textwidth,angle=0]{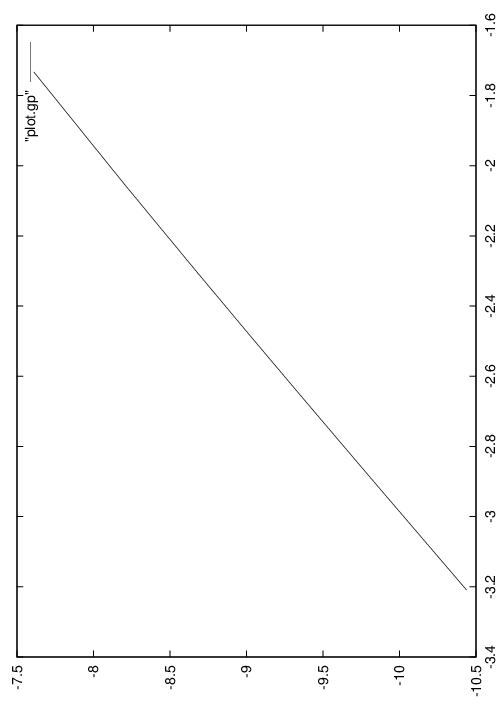}
		\end{center}
		\caption{\footnotesize{Error for the velocity $\parallel \textbf{u}-\textbf{u}_h\parallel_h$ in $\Omega_s$ ( log/log plot)}}
		\label{EVS}
	\end{minipage}
	\begin{minipage}[c]{.49\textwidth}
		\centering
		\begin{center}
			\includegraphics[width=.6\textwidth]{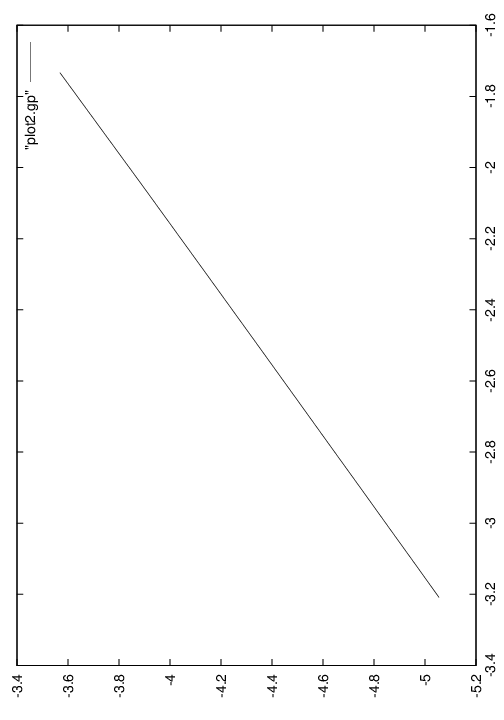}
		\end{center}
		\caption{\footnotesize{Error for the pressure $\parallel p-p_h\parallel $ in $\Omega_s$ (log/log plot)}}
		\label{EPS}
	\end{minipage}
\end{figure}
\begin{figure}[http]
	\begin{minipage}[c]{.49\textwidth}
		\centering
		\begin{center}
			\includegraphics[width=.6\textwidth,angle=0]{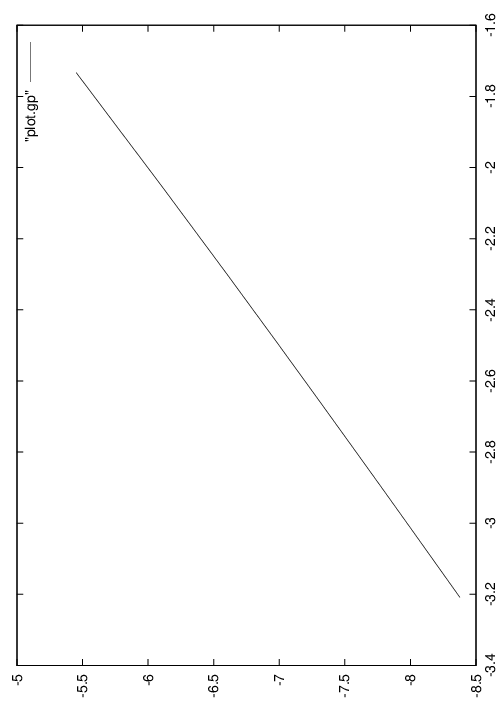}
		\end{center}
		\caption{\footnotesize{Error for the velocity $\parallel\textbf{u}-\textbf{u}_h\parallel_h$ in $\Omega_d$ (log/log plot)}}
		\label{EVD}
	\end{minipage}
	\begin{minipage}[c]{.49\textwidth}
		\centering
		\begin{center}
			\includegraphics[width=.6\textwidth]{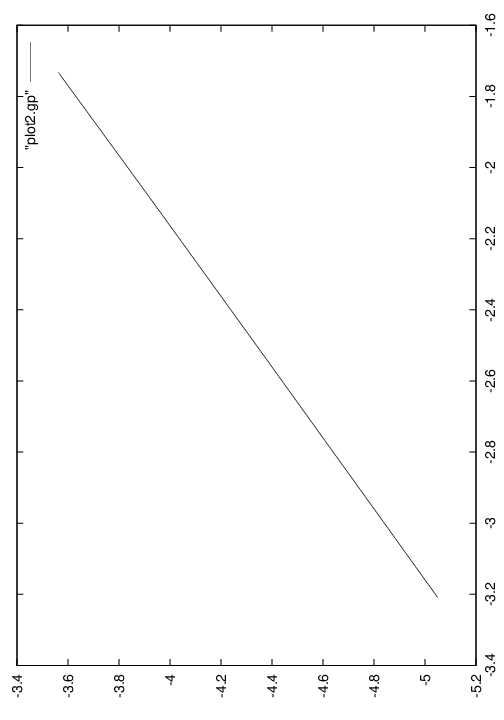}
		\end{center}
		\caption{\footnotesize{Error for the pressure $\parallel p-p_h\parallel$ in $\Omega_d$ (log/log plot)}}
		\label{EPD}
	\end{minipage}
\end{figure}
\begin{figure}[http]
	\begin{minipage}[c]{.49\textwidth}
		\centering
		\begin{center}
			\includegraphics[width=.8\textwidth]{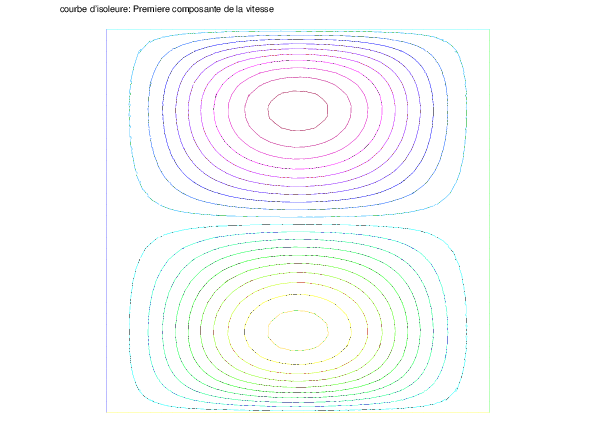}
		\end{center}
		\caption{\footnotesize{The isovalue of the first velocity component $u_1$ in $\Omega_s$.}}
	\end{minipage}
	\begin{minipage}[c]{.49\textwidth}
		\centering
		\begin{center}
			\includegraphics[width=.8\textwidth]{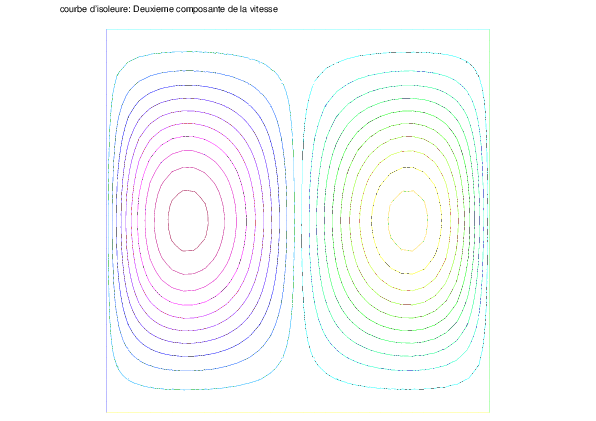}
		\end{center}
		\caption{\footnotesize{The isovalue of the second velocity component $u_2$ in $\Omega_s$}}
	\end{minipage}
\end{figure}

\begin{figure}[http]
	\begin{minipage}[c]{.49\textwidth}
		\centering
		\begin{center}
			\includegraphics[width=.8\textwidth]{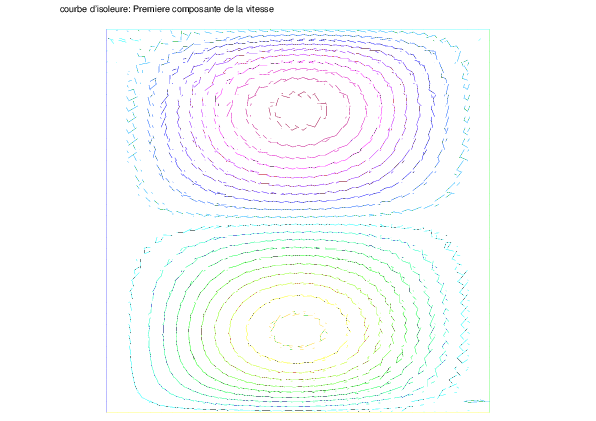}
		\end{center}
		\caption{\footnotesize{The isovalue of the first velocity component $u_1$ in $\Omega_d$}}
	\end{minipage}
	\begin{minipage}[c]{.49\textwidth}
		\centering
		\begin{center}
			\includegraphics[width=.8\textwidth]{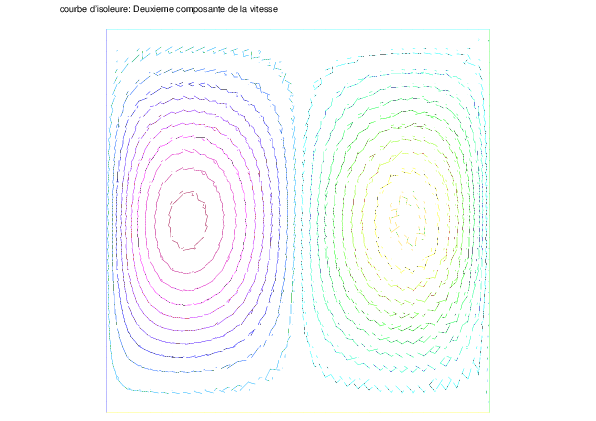}
		\end{center}
		\caption{\footnotesize{The isovalue of the second velocity component $u_2$ in $\Omega_d$ }}
	\end{minipage}
\end{figure}

\begin{table}[htbp]
	\begin{center}
		\begin{tabular}{|l|l|}
			\hline 
			Parameter $h$ & $^o/_{o}$ of Coefs. $\neq$ 0 \\
			\hline
			1/16 & 79.668827 \\
			\hline 
			1/22 & 66.405634\\
			\hline
			1/28 & 41.751001\\
			\hline
			1/34 & 28.652204\\
			\hline
			1/40 & 20.873193\\
			\hline
			1/46 & 15.879941\\
			\hline
			1/52 & 12.485349\\
			\hline 
			1/58 & 10.073294\\
			\hline 
			1/66 & 08.298201\\
			\hline
			1/70 & 06.954061\\
			\hline
		\end{tabular}
	\end{center}
	\caption{ Structure of rigidity Matrix on $10$ iterations.}
	\label{Rigidite} 
\end{table}
\begin{figure}[http]
	\begin{minipage}[c]{.49\textwidth}
		\centering
		\begin{center}
			\includegraphics[height= 4cm, width=4cm]{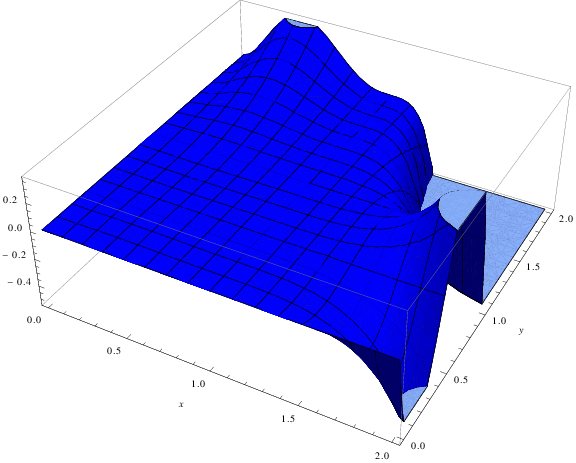}
		\end{center}
		\caption{\footnotesize{Component $u_1$ in $\Omega$.}}
	\end{minipage}
	\begin{minipage}[c]{.49\textwidth}
		\centering
		\begin{center}
			\includegraphics[height= 4cm, width=4cm]{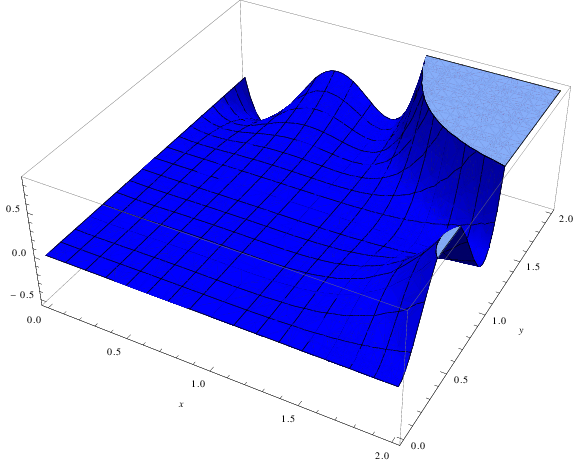}
		\end{center}
		\caption{\footnotesize{Component $u_2$ in $\Omega$.}}
	\end{minipage}
\end{figure}

\begin{figure}[http]
	\centering
	\begin{center}
		\includegraphics[height= 4cm, width=4cm]{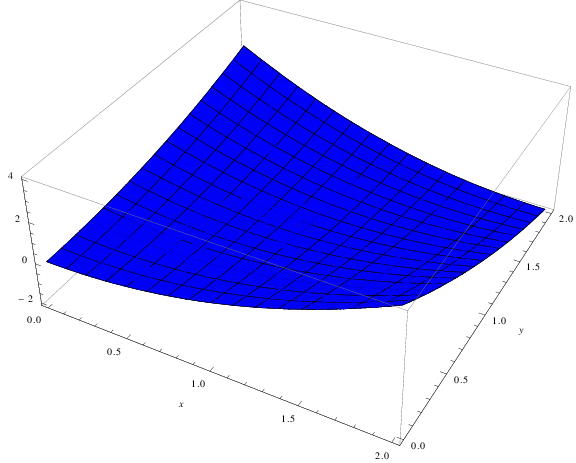}
	\end{center}
	\caption{\footnotesize{Pressure $p$ in $\Omega$.}}
\end{figure}
\begin{figure}[http]
	\begin{minipage}[c]{.49\textwidth}
		\centering
		\begin{center}
			\includegraphics[height= 4cm, width=4cm]{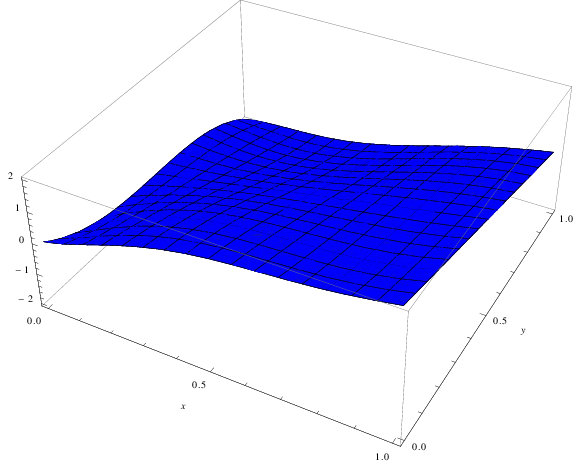}
		\end{center}
		\caption{\footnotesize{Right-hand term $f_1$ in $\Omega_s$.}}
	\end{minipage}
	\begin{minipage}[c]{.49\textwidth}
		\centering
		\begin{center}
			\includegraphics[height= 4cm, width=4cm]{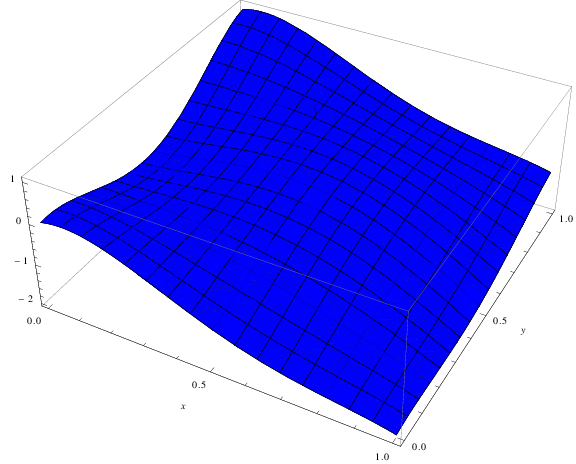}
		\end{center}
		\caption{\footnotesize{Right-hand term $f_2$ in $\Omega_s$.}}
	\end{minipage}
\end{figure}
\begin{figure}[http]
	\begin{minipage}[c]{.49\textwidth}
		\centering
		\begin{center}
			\includegraphics[height= 4cm, width=4cm]{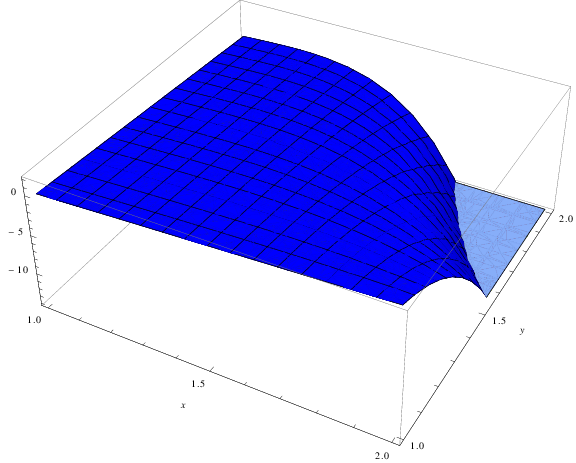}
		\end{center}
		\caption{\footnotesize{Right-hand term $k_1$ in $\Omega_d$.}}
	\end{minipage}
	\begin{minipage}[c]{.49\textwidth}
		\centering
		\begin{center}
			\includegraphics[height= 4cm, width=4cm]{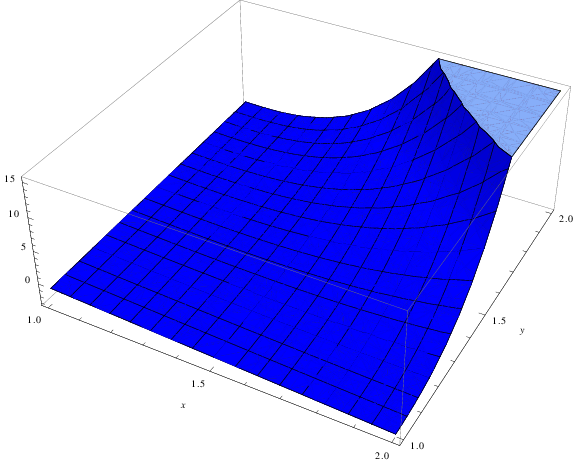}
		\end{center}
		\caption{\footnotesize{Right-hand term $k_2$ in $\Omega_d$.}}
	\end{minipage}
\end{figure}
\section{Conclusion}
In this contribution, we investigated a new mixed finite element method to solve the Stokes-Darcy fluid flow model without introducing any Lagrange multiplier. We proposed a modification of the Darcy problem which allows us to apply a slight variant nonconforming Crouzeix-Raviart element to the whole coupled Stokes-Darcy problem. The proposed method is probably one the cheapest method for Discontinuous Galerkin $(DG)$ approximation of the coupled system, has optimal accuracy with respect to solution regularity, and has simple and straightforward implementations. Numerical experiments have been also presented, which confirm the excellent stability and accuracy of our method.
\section{Acknowledgment}
The author thanks Professor Emmanuel Creus\'e (University of  Lille 1, France) for having sent us useful documents and for fruitful discussions concerning the numerical tests.

\end{Large}

\begin{thebibliography}{10}
\bibitem{21}
T.~Arbogast and D.~Brunson.
\newblock A computational method for approximating a {D}arcy-{S}tokes system
governing a vuggy porous medium.
\newblock {\em Computational Geosciences}, 11:207--218, 2007.

\bibitem{GL:2018}
M.~G. Armentano and M.~L. Stockdale.
\newblock A unified mixed finite element approximations of the {S}tokes-{D}arcy
coupled problem.
\newblock {\em Computers and Mathematics with Applications},
https://doi.org/10.1016/j.camwa.2018.12.032, 2018.

\bibitem{49}
I.~Babu\v{s}ka and G.~Gatica.
\newblock A residual-based a posteriori error estimator for the
{S}tokes-{D}arcy coupled problem.
\newblock {\em SIAM J. Numer. Anal.}, 48:498--523, 2010.

\bibitem{23}
G.~Beavers and D.~Joseph.
\newblock Boundary conditions at a naturally permeable wall.
\newblock {\em J. Fluid Mech.}, 30:197--207, 1967.

\bibitem{B:03}
S.~Brenner.
\newblock Korn's inequalities for piecewise ${H}^1$ vector fields.
\newblock {\em Math. Comput.}, 73:1067--1087, 2003.

\bibitem{43}
W.~Chen and Y.~Wang.
\newblock A posteriori error estimate for {H}(div) conforming mixed finite
element for the coupled {D}arcy-{S}tokes system.
\newblock {\em Journal of Computational and Applied Mathematics}, 255:502--516,
2014.

\bibitem{CR:73}
M.~Crouzeix and P.~Raviart.
\newblock Conforming and nonconforming finite element methods for solving the
stationary {S}tokes equations.
\newblock {\em Rev. Fran\c{c}aise Automat. Informat. Recherche Op\'erationnelle
	s\'er. Rouge}, 7:33--75, 1973.

\bibitem{46}
M.~Cui and N.~Yan.
\newblock A posteriori error estimate for the {S}tokes-{D}arcy system.
\newblock {\em Math. Meth. Appl. Sci.}, 34:1050--1064, 2011.

\bibitem{27}
M.~Discacciati and A.~Quarteroni.
\newblock Navier-{S}tokes/{D}arcy coupling: Modeling, analysis, and numerical
approximation.
\newblock {\em Rev. Math. Comput.}, 22:315--426, 2009.

\bibitem{AE05}
A.~Ern.
\newblock Aide-m\'emoire el\'ements finis.
\newblock {\em Dunod, Paris, ISBN 2 10 007303 6}, 2005.

\bibitem{HP:FreeFem}
H.~Frederic and P.~Olivier.
\newblock Freefem++.
\newblock {\em http://www.freefem.org}.

\bibitem{45}
J.~Galvis and M.~Sarkis.
\newblock Nonconforming mortar discretization analysis for the coupling
{S}tokes-{D}arcy equations.
\newblock {\em Electronic. Trans. Numer. Anal.}, 26:350--384, 2007.

\bibitem{47}
G.~Gatica, R.~Oyarz\`ua, and F.-J. Sayas.
\newblock A residual-based a posteriori error estimator for a fully-mixed
formulation of the {S}tokes-{D}arcy coupled problem.
\newblock {\em Comput. Methods Appl. Mech. Engrg.}, 200:1877--1891, 2011.

\bibitem{29}
G.~N. Gatica, S.~Meddahi, and R.~Oyarz\`ua.
\newblock A conforming mixed finite element method for the coupling of fluid
flow with porous media flow.
\newblock {\em IMA J. Numer. Anal.}, 29:86--108, 2009.

\bibitem{30}
G.-N. Gatica, R.~Oyarz\`ua, and F.-J. Sayas.
\newblock Convergence of a family of galerkin discretizations for the
{S}tokes-{D}arcy coupled proplem.

\bibitem{girault:86}
V.~Girault and P.-A. Raviart.
\newblock {\em Finite element methods for {N}avier-{S}tokes equations, {T}heory
	and algorithms}, volume~5 of {\em Springer, Berlin}.
\newblock In Computational Mathematics, 1986.

\bibitem{FH:98}
F.~Hecht.
\newblock The mesh adapting software: bamg.
\newblock {\em INRIA
	report.http://www-c.inria.fr/gamma/cdrom/www/bamg/eng.htm}, 1998.

\bibitem{HJA:2017}
K.~W. Hou\'edanou, J.~Adetola, and B.~Ahounou.
\newblock Residual-based a posteriori error estimates for a conforming finite
element discretization of the {N}avier-{S}tokes/{D}arcy coupled problem.
\newblock {\em Journal of Pure and Applied Mathematics : Advances and
	Applications}, 18(1):37--73, 2017.

\bibitem{HA:2016}
K.~W. Hou\'edanou and B.~Ahounou.
\newblock A posteriori error estimation for the {S}tokes-{D}arcy coupled
problem on anisotropic discretization.
\newblock {\em Math. Meth. Appl. Sci.}, 2016.
\newblock http://dx.doi.org/10.1002/mma.4261 (in press).

\bibitem{32}
W.~J\"{a}ger and A.~Mikeli\'{c}.
\newblock On the boundary conditions of the contact interface between a porous
medium and a free fluid.
\newblock {\em Ann. Scuola Norm. Sup. Oisa Cl. Sci.}, 23:403--465, 1996.

\bibitem{17}
W.~J\"{a}ger and A.~Mikeli\'{c}.
\newblock On the interface boundary condition of {B}eavers, {J}oseph and
{S}affman.
\newblock {\em SIAM Journal on Applied Mathematics}, 60:1111--1127, 2000.

\bibitem{JM:2000}
W.~J\"{a}ger and A.~Mikeli\'{c}.
\newblock On the interface boundary condition of {B}eavers, {J}oseph and
{S}affman.
\newblock {\em {SIAM} J. Appl. Math.}, 60:1111--1127, 2000.

\bibitem{48}
W.~J\"{a}ger, A.~Mikeli\'{c}, and N.~Neuss.
\newblock Asymptotic analysis of the laminar visous flow over a porous bed.
\newblock {\em {SIAM} J. Sci. Comput.}, 22:2006--2028, 2001.

\bibitem{RJZY:2015}
R.~Li, J.~Li, Z.~Chen, and Y.~Gao.
\newblock A stabilized finite element method based on two local gauss
integrations for a coupled {S}tokes-{D}arcy problem.
\newblock {\em http://dx.doi.org/10.1016/j.cam.2015.06.014}, 2015.

\bibitem{18}
K.-A. Mardal, X.~Tai, and R.~Winther.
\newblock A robust finite element method for darcy-stokes flow.
\newblock {\em SIAM Journal on Numerical Analysis}, 40:1605--1631, 2002.

\bibitem{19}
M.~Mu and J.~Xu.
\newblock A two-grid method of a mixed {S}tokes-{D}arcy model for coupling
fluid flow with porous media flow.
\newblock {\em SIAM Journal on Numerical Analysis}, 45:1801--1813, 2007.

\bibitem{AHN:15}
S.~Nicaise, B.~Ahounou, and W.~Hou\'edanou.
\newblock A residual-based posteriori error estimates for a nonconforming
finite element discretization of the {S}tokes-{D}arcy coupled problem:
Isotropic discretization.
\newblock {\em Afr. Mat., African Mathematical Union and Springer-Verlag Berlin
	Heidelberg: New York}, 27(3):701--729 (2016), 2015.

\bibitem{37}
L.~Payne and B.~Straughan.
\newblock Analysis of the boundary condition at the interface between a viscous
fluid and a porous medium and related modeling questions.
\newblock {\em J. Math. Pures Appl.}, 77:317--354, 1998.

\bibitem{39}
B.~Rivi\`ere and I.~Yotov.
\newblock Locally conservative coupling of {S}tokes and {D}arcy flows.
\newblock {\em SIAM J. Numer. Anal.}, 42:1959--1977, 2005.

\bibitem{7}
H.~Rui and R.~Zhang.
\newblock A unified stabilized mixed finite element method for coupling
{S}tokes and {D}arcy flows.
\newblock {\em Comput. Methods Appl. Mech. Engrg.}, 198:2692--2699, 2009.

\bibitem{44}
P.~Saffman.
\newblock On the boundary condition at the interface of a porous medium.
\newblock {\em Stud. Appl. Math.}, 1:93--101, 1971.

\bibitem{DI:09}
D.~Vassilev and I.~Yotov.
\newblock Coupling {S}tokes-{D}arcy flow with transport.
\newblock {\em SIAM J. Sci. Comput.}, 5:3661--3684, 2009.

\bibitem{12}
L.~J. William, S.~Friedhelm, and Y.~Ivan.
\newblock Coupling fluid flow with porous media flow.
\newblock {\em SIAM J. Numer. Anal.}, 40(6):2195--2218 (2003), 2002.

\bibitem{JAFH:2018}
J.~Yu, M.~A.~A. Mahbub, F.~Shi, and H.~Zheng.
\newblock Stabilized finite element method for the stationary mixed
{S}tokes-{D}arcy problem.
\newblock {\em Advances in Diffference Equations}, https://
doi.org/10.1186/s13662-018-1809-2346, 2018.

\end{thebibliography}
\end{document}